\documentclass{amsart}
\usepackage{graphicx} % Required for inserting images

\usepackage{amscd,amsmath,amssymb,amsfonts,verbatim}
\usepackage[cmtip, all]{xy}

\usepackage{MnSymbol}
\usepackage{comment}
\usepackage{url}

\newtheorem{thm}{Theorem}[section]
\newtheorem{prop}[thm]{Proposition}
\newtheorem{lem}[thm]{Lemma}
\newtheorem{cor}[thm]{Corollary}

\theoremstyle{definition}

\newtheorem{defn}[thm]{Definition}
\theoremstyle{remark}
\newtheorem{remk}[thm]{Remark}
\newtheorem{remks}[thm]{Remarks}

\newtheorem{exm}[thm]{Example}
\newtheorem{exms}[thm]{Examples}
\newtheorem{notat}[thm]{Notation}
\numberwithin{equation}{section}

{\hfill$\square$\end{defn}}
\newenvironment{rem}{\begin{remk}}%
{\hfill$\square$\end{remk}}
{\hfill$\square$\end{remks}}
{\hfill$\square$\end{exm}}
{\hfill$\square$\end{exms}}
{\hfill$\square$\end{notat}}

\newcommand{\thmref}{Theorem~\ref}

\newcommand{\Q}{{\mathbb Q}}

\newcommand{\Z}{{\mathbb Z}}

\newcommand{\surj}{\twoheadrightarrow}
\newcommand{\inj}{\hookrightarrow}

\newcommand{\ab}{{\rm ab}}
\newcommand{\tor}{{\rm tor}}
\newcommand{\geom}{{\rm geom}}
\newcommand{\Ker}{{\rm Ker}}
\newcommand{\Coker}{{\rm Coker}}

\newcommand{\Spec}{{\rm Spec}}
\newcommand{\Tr}{{\rm Tr}}

\title[Finiteness of torsion subgroup]{Torsion in abelian fundamental group and its application}
\author{Rahul Gupta}
\address{Institute of Mathematical Sciences, A CI of Homi Bhabha National Institute, 4th Cross St., CIT Campus, Tharamani, Chennai,
  600113, India.} 
  \email{rahulgupta@imsc.res.in}
\author{Jitendra Rathore}
\address{Department of Mathematics, University of California, Santa Barbara, CA 93106, USA.}
\email{jitendra@ucsb.edu}

\date{\today}

\keywords{Abelian \'etale fundamental group, class field theory, local fields}

\subjclass[2020]{Primary 14G20; Secondary 14H30, 14C35.}

\begin{document}

\begin{abstract}
    We prove that the torsion subgroup of the abelian fundamental group is finite for a regular geometrically integral projective variety over a local field. We also study the structure of $SK_1(X)$ for a regular projective variety $X$ over a local field. As an application, we get class field theory for regular projective curves over local fields. 
\end{abstract}

\maketitle

\setcounter{tocdepth}{1}
%\maketitle
\tableofcontents

\section{Introduction} \label{sec:Intro}

	 The {\'e}tale fundamental group of an integral normal variety plays an important role in studying various problems in arithmetic geometry.  The maximal abelian profinite quotient of this group, known as the abelianized {\'e}tale fundamental group, is one of the main objects of study in geometric class field theory. This group
	has been extensively studied for smooth projective varieties 
       over finite and local fields throughout the decades (see \cite{KL81}, \cite{Bloch-Annals}, \cite{KS-UCFT}, \cite{Saito85}, \cite{Jannsen-Saito}, \cite{Yos03}, \cite{Forre-Crelle}, \cite{Hira-I}, and  \cite{TCFT-arXiv}).
       In this paper, we study this group for a regular, not necessarily smooth, projective variety over a local field of positive characteristic.
	
	Let $X$ be an integral normal scheme over a field $k$. We denote by $\pi^{\ab}_{1}(X)$ the abelianized {\'e}tale fundamental group of $X$. Recall that $\pi^{\ab}_{1}(X)$  classifies all finite abelian {\'e}tale coverings over $X$.  The objective of this paper is to study the torsion subgroup of this group for a regular geometrically integral quasi-projective variety over a local field. 
	We let $ \pi^{\ab}_{1}(X)^{\geom} : = \Ker ( \pi^{\ab}_{1}(X) \rightarrow \pi^{\ab}_{1}(\Spec(k)) )$, which is called the geometric part of the group $\pi^{\ab}_{1}(X)$. 
	In the case where $k$ is finite or local field, the local class field theory provides an explicit structure for the group $\pi^{\ab}_{1}(\Spec(k))$ and therefore the study of  $\pi^{\ab}_{1}(X)$ in these cases is reduced to the study of the group  $ \pi^{\ab}_{1}(X)^{\geom}$.
	
	Let $k$ be a finite field of characteristic $p$ and let $X$ be a smooth geometrically connected projective scheme over $k$. Then the group $ \pi^{\ab}_{1}(X)^{\geom}$ is known to be finite by Katz-Lang \cite{KL81}. More generally, if $X$ is smooth geometrically connected finite type scheme over $k$, then $ \pi^{\ab}_{1}(X)^{\geom}$ is finite, except for its pro-$p$ subgroup.

	Let us now assume that $k$ is a local field, i.e., a complete, discrete-valued field with finite residue field. Let $X$ be a smooth geometrically connected projective scheme over $k$. Then the group $ \pi^{\ab}_{1}(X)^{\geom}$ is not necessarily finite but it is known to be a topologically finitely generated profinite group. More precisely, we have 
	$ \pi^{\ab}_{1}(X)^{\geom} \cong F \oplus \widehat{\mathbb{Z}}^{r}$, where $F$ is a finite group and $r\geq 0$. This structure theorem was proven by Grothendieck \cite{SGA1} and Yoshida \cite{Yos03}. This result was recently extended to 
	smooth quasi-projective varieties having smooth compactifications (except for the pro-$p$ subgroup) \cite [Theorem 1.1]{TCFT-arXiv}. 
	These results, in particular, imply that the torsion subgroup $\pi^{\ab}_{1}(X)_{{\rm tor}}$ of the abelian fundamental group is finite for smooth projective varieties and the torsion subgroup except the $p$-primary torsion is finite in the case of smooth quasi-projective varieties.

	The main objective of this paper is to extend these finiteness results to the case of regular (not necessarily smooth), projective (possibly quasi-projective), geometrically connected varieties over local fields of positive characteristic. Such questions do not arise over finite fields, as regular quasi-projective schemes over finite fields are smooth. One of the consequences of the finiteness of the group  $\pi^{\ab}_{1}(X)_{{\rm tor}}$ is that it plays a key role in the study of class field theory for $X$ (see \cite{Jannsen-Saito}, \cite{Forre-Crelle}, \cite{Hira-I}, and  \cite{TCFT-arXiv}). As an application of such a finiteness result for  regular varieties, we extend the results of class field theory for smooth projective curves to regular projective %(possibly quasi-projective)
	curves over local fields.
	
	We now state our main results.
	
	\begin{thm} \label{thm:A}
		Let $k$ be a local field of positive characteristic $p >0$ and let $X$ be a geometrically integral scheme over $k$.
		\begin{enumerate}
			\item If $X$ is regular and projective, then  $\pi^{\ab}_{1}(X)_{{\rm tor}}$ is finite.
			
			\item  If $X$ is regular and quasi-projective, then  $\pi^{\ab}_{1}(X)_{{\rm tor}}$ (except for  $p$-torsion) is finite.
		\end{enumerate}
		\end{thm}

	\subsection{Application}
	As mentioned earlier, the finiteness of the group  $\pi^{\ab}_{1}(X)_{{\rm tor}}$ plays
    a key role in the study of class field theory of $X$. In this light, 
	we apply Theorem~\ref{thm:A} to prove results in unramified class field theory for regular projective curves over local fields. To do this, we first study the group $SK_{1}(X)$ for an integral projective scheme $X$ over a local field $k$, which is defined as 
	\begin{center}
		$SK_{1}(X): = {{\rm coker }}   \Big(\oplus_{x \in X_{(1)}} K_{2}(k(x)) \rightarrow \oplus_{x \in X_{(0)}} K_{1}(k(x)) \Big)$.
	\end{center}
	We let	$V(X/k): = {{\rm Ker }}  \Big ( SK_{1}(X)  \xrightarrow{{{\rm Norm}}} k^{\times} \Big)$. These groups play a key role in studying the class field theory for smooth projective varieties over local fields via the reciprocity map $\rho_X \colon SK_1(X) \to \pi_1^{\ab}(X)$ and the induced map  $\rho_{X,0} \colon V(X/k) \to \pi^{\ab}_1(X)^{{\rm geom}}$.
	
	These groups were introduced by Bloch \cite{Bloch-Annals} to study the unramified class field theory for a certain class of regular schemes.
	It is shown in loc. cit. that when $X$ is an elliptic curve over a local field, we have $V(X/k) \cong T \oplus D$, where $T$ is a torsion group and $D$ is a divisible group  (\cite[Corollary 2.38]{Bloch-Annals}). This was later generalized to smooth projective curves and further extended to smooth projective varieties over local fields (see \cite{Saito85}, \cite{Yos03}, \cite{Jannsen-Saito}, and \cite{Forre-Crelle}).
 We extend this result to geometrically irreducible schemes over local fields of characteristic $p >0$ and prove the following.
	\begin{thm} \label{thm:B'}
		Let $X$ be a geometrically irreducible projective integral scheme of dimension $d \geq 0$ over a local field $k$ of positive characteristic $p >  0$. Then 
		\begin{equation} \label{eqn:thm-B'*-1}
			V(X/k) = T \oplus D,
		\end{equation}
		where $T$ is a torsion group and $D$ is a divisible group. 
	\end{thm}

 The class field theory for smooth projective curve $X$ over a local field was first studied by Bloch \cite{Bloch-Annals} when $X$ has good reduction and this was later generalized to smooth projective geometrically connected curves by Saito \cite{Saito85} (prime to $p$-part) and by Yoshida \cite{Yos03} ($p$-part). Here, we extend these results to integral regular projective curves.  
	
	We let $k$ denote a local field of positive characteristic $p>0$ and let $X$ be a geometrically integral regular projective curve over $k$. We then have a reciprocity map (see \cite{Saito85}, 
	\cite[Section 1.5]{KS83} or \cite[Proof of Lemma~7.2]{TCFT-arXiv})
	\begin{equation} \label{eqn:Rec-map*-1}
		\rho_X \colon SK_1(X) \to \pi_1^{\ab}(X). 
	\end{equation}
	We let $\rho_{X,0} \colon V(X/k) \to \pi^{\ab}_1(X)^{{\rm geom}}$ denote the induced map. We then have the following result. 
	
	\begin{thm}\label{thm:CF-Reg}
		Let $X$ be as above. Then 
		\begin{enumerate}
			\item $\Ker(\rho_X)$ is divisible by every integer prime to $p$. Moreover, if we assume further that $X$ is generically smooth, then $\Ker(\rho_X)$ is the maximal divisible subgroup of $SK_{1}(X)$.
			\item ${\rm Image}(\rho_{X,0})$ is a finite group. 
            \item  $\Coker(\rho_X)_{\tor}$ is a finite group.
		\end{enumerate}
	\end{thm}

\subsection{Idea of the proofs}

We prove \thmref{thm:A} in two steps. We first prove that the prime to $p$-torsion subgroup $\pi^{\ab}_1(X)\{p'\}$ is finite. The main ingredient to prove this is to show that for a dense open subset $U \subset X$,  
$\pi^{\ab}_1(X)\{p'\}$ is finite if and only if $\pi^{\ab}_1(U)\{p'\}$ is finite. We then use an alteration of degree $p^r$ of $X$ and the smooth case to get hold of a dense open subset $U$ such that $\pi^{\ab}_1(U)\{p'\}$ is finite. 
The second and final step is to show that the $p$-torsion subgroup 
$\pi^{\ab}_1(X)\{p\}$ is finite (see Theorem~\ref{thm:Main-thm-2}). 
 Since the abelian absolute Galois group $G^{\ab}_k$ has no $p$-torsion for local fields, Theorem~\ref{thm:Main-thm-2} follows if we show that the $p$-primary torsion subgroup of $\pi^{\ab}_1(X)^{\geom}$ is finite. This is achieved in Proposition~\ref{prop:Struc-geom-*1}.

We first reduce Theorem \ref{thm:B'} to the case when $X$ is a geometrically irreducible projective integral curve in Section~\ref{sec:Red-to-curve}. We then use the fact that for such a curve, there is a nice alteration map to get a stronger result where we can assume that torsion summand $T$ is of finite exponent (see Theorem~\ref{thm:B-curve}). 

In Section~\ref{sec:CFT}, we prove duality results for regular projective curves and use them to show that the reciprocity map modulo $n$ is injective for all $n \in \mathbb{Z}$, see Theorems~\ref{thm:Res-inj-mod-n} and~\ref{thm:Res-inj-mod-p}. As an application of Theorem~\ref{thm:A}, we then obtain Theorem~\ref{thm:CF-Reg} in Section~\ref{sec:Proof-CFT}.

\subsection{Notations} \label{sec:Not}

For an abelian group $A$ and $n\geq 1$, we let $A[n]$ denote the kernel of the multiplication map $A \xrightarrow{n} A$. For a prime $p$, we let the subgroup of prime to $p$-primary torsion $A\{p'\}:=  \cup_{{\rm gcd}(p, n)=1} A[n]$ and let 
$A\{p\}$ denote the $p$-primary torsion subgroup $\cup_{i}A[p^i]$. Similarly, for a fixed prime $p$, we say $A$ is $p'$-divisible (resp. $p$-divisible) if $A$ is divisible by all $n$ such that ${\rm gcd}(p, n)=1$ (resp. divisible by $p^i$ for $i\geq 1$). 
For a torsion group $A$ of finite exponent, we abuse the notation and say that it is of exponent $m$ if $mA =0$.  

For a field $k$, we let $G_k$ denote its absolute Galois group ${{\rm Gal}}(k^{{\rm sep}}/k)$, where $k^{{\rm sep}}$ is a separable closure of $k$. 
For a finite type scheme $X$ over a field and $i \geq 0$, we denote the set of $i$-dimensional points of $X$ by $X_{(i)}$. For $y \in X_{(i)}$, we consider the closed subset $\overline{\{y\}} \subset X$ with its reduced induced structure and denote the resulting $i$-dimensional integral closed subscheme by $C_y$. If $X$ is integral, we let $X^N \to X$ denote the normalization of $X$. 
For a scheme $X \to \Spec(k)$ and a field extension $l/k$, we let $X\times_k l $ denote the base change scheme $X \times_{\Spec(k)} \Spec(l)\to \Spec(l)$. Similarly, we use the notation $\phi \times_k l \colon Y\times_k l \to X\times_k l$ for the base change of morphism $\phi \colon Y \to X$ between schemes $X$ and $Y$ over $k$.

All cohomology groups considered in the paper are \'etale cohomology groups. 

\section{{\'e}tale cohomology groups} \label{sec:et-coho}

Let $k$ be a field of exponential characteristic $p \geq 1$ and let $X$ be a quasi-projective scheme over $k$. Recall that for $m \geq 1$ such that 
${{\rm gcd}}(p,m)=1$, the presheaf $\mu_{m, X}$ of $m$-th roots of unity is an \'etale sheaf on $X$. For $j \geq 0$, we let 
$\Z/m(j)_X = \mu_{m,X}^{\otimes j}$ and for $j<0$, we let
$\Z/m(j)= {\mathcal{H}om}(\Z/m(-j)_X, \Z/m)$. For a prime $l\neq p$ and $i, j \in \Z$, we let $H^i(X, \Q_l/\Z_l(j)) = \varinjlim_{n} H^i(X, \Z/l^n(j))$ and let 
\begin{equation} \label{eqn:def-et-coho}
    H^i(X, (\Q/\Z)'(j)) := \oplus_{l \neq p} H^i(X, \Q_l/\Z_l(j)),
\end{equation}
where $l$ runs over all prime integers different from $p$. 
For a closed subset $Z\subset X$, we similarly define support \'etale cohomology groups 
$ H^i_Z(X, (\Q/\Z)'(j))$. 
We begin by recalling the following result from \cite[\S~8]{Fuji02}.

\begin{lem} \label{lem:Semi-Purity}
    Let $Y$ be a regular finite type scheme over a field $k$ of dimension $d\geq 1$ and let $W\subset Y$ be a closed subset of codimension $c\geq 1$. Let $m$ be an integer that is invertible in $k$ and let $j\in \Z$.  We then have 
    \begin{equation} \label{eqn:SP-*0}
        H^i_W(Y, \Z/m(j) )= 0  \textnormal{ for all $i< 2c$}. 
    \end{equation}
\end{lem}

Using Lemma \ref{lem:Semi-Purity} and the support-cohomology long exact sequence, we immediately get the following result. 

\begin{cor}\label{cor:Semi-purity}
    Let $Y$ be a regular finite type scheme over a field $k$ of dimension $d\geq 1$ and let $U\subset Y$ be an open subset such that its complement $Y \setminus U$ has codimension $c\geq 1$. Then the pull-back 
    $H^i(Y, (\Q/\Z)'(j)) \to H^i(U, (\Q/\Z)'(j)) $ is an isomorphism for all $i \leq 2c-2$ and $j \in \Z$. 
\end{cor}

We use the above corollary to prove the finiteness of certain \'etale cohomology groups.

\begin{lem} \label{lem:et-como-finite*-1}
 Let $k$ be a local field and let $Z$ be a regular quasi-projective variety over $k$. Then the group $H^0(Z, (\Q/\Z)'(-1))$ is finite. 
\end{lem}
\begin{proof}
    If $Z$ is regular projective variety over a local field, then the lemma follows from \cite[Lemma~10.9]{TCFT-arXiv}. Indeed, in the loc. cit., they proved that for a projective scheme $W$ of dimension $d \geq0$ over a local field and $i \geq 0$, the group $H^i(W, (\Q/\Z)'(j))$ is finite for all  $j \leq 0$ and for all $j \geq d+2$. 

    We now let $Z$ be a regular quasi-projective integral variety over $k$ and let 
    $Z \inj \overline{Z}$ be a compactification of $Z$.  Let $\pi\colon W \to \overline{Z}$ be an alteration morphism of de Jong and Gabber \cite{Tem17}. In particular, the morphism $\pi$ is generically finite of degree $e$, where $e =1$ if ${{\rm Char}}(k)=0$ and $e=p^r$ if ${{\rm Char}}(k)=p>0$. (Note that by a generically finite map of degree $1$, we mean a birational map.) 
    Let $V \inj \overline{Z}$ be a dense open subset of $\overline{Z}$ such that the induced map $\pi_{V'}\colon V' = \pi^{-1}(V) \to V$ is finite and flat of degree $e$. Since $Z$ is integral, the intersection $U = V \cap Z$ is a dense open subset of $Z$. We let $\pi_{U'} \colon U'=\pi^{-1}(U)\to U$ denote the induced map. Since $\pi_{V'}$ is finite and flat of degree $e$, it follows that the morphism $\pi_{U'}$ is also finite and flat of degree $e$. We consider the following diagram. 
\begin{equation}\label{eqn:finiteness*-1.1}
    \xymatrix@C.8pc{
          H^0(Z, (\Q/\Z)'(-1)) \ar[r]^-{\cong} & H^0(U, (\Q/\Z)'(-1)) \ar[d]^-{\pi_{U'}^*} \\
         H^0(W, (\Q/\Z)'(-1)) \ar[r]^-{\cong}& H^0(U', (\Q/\Z)'(-1)).}
\end{equation}    
Note that the horizontal arrows are isomorphisms by Corollary \ref{cor:Semi-purity} (with $c \geq 1$) because $U$ (resp. $U'$) is a dense open subset of $Z$ (resp. of $W$), and both $Z$ and $W$ are regular finite type schemes over $k$. Since $W$ is a regular projective scheme over a local field, it follows from the first paragraph that the bottom left group $H^0(W, (\Q/\Z)'(-1))$ is a finite group. It therefore suffices to show that the right vertical arrow $\pi_{U'}^* \colon  H^0(U, (\Q/\Z)'(-1)) \to  H^0(U', (\Q/\Z)'(-1))$ is injective. But this follows from \cite[XVII, Proposition~6.2.5]{SGA4}. Indeed, by the loc. cit., we have a map 
$(\pi_{U'})_*\colon H^0(U', (\Q/\Z)'(-1)) \to H^0(U, (\Q/\Z)'(-1))$ such that 
$(\pi_{U'})_* \circ \pi_{U'}^* = {{\rm deg}}(\pi_{U'}) {{\rm Id}} = p^r 
{{\rm Id}}$. Since multiplication by $p$ is an isomorphism on  \'etale cohomology with $(\Q/\Z)'(j)$-coefficients, it follows that the composition map $(\pi_{U'})_* \circ \pi_{U'}^*$ is an isomorphism. In particular, the right vertical arrow $\pi_{U'}^*$ in \eqref{eqn:finiteness*-1.1} is injective. This completes the proof. 
\end{proof}

We now prove the main result of the section. We use this result later to compare the abelian \'etale fundamental group of regular schemes over local fields with that of its dense open subsets. 
%Note that the following lemma improves Lemma \ref{lem:Semi-Purity}. 

\begin{lem} \label{lem:et-coho-supp-fin-*2}
    Let $X$ be a regular quasi-projective scheme of dimension $d\geq 1$ over a local field $k$ and let $Z \subset X$ be a closed subset of codimension $c=1$. Then the cohomology group $H^2_Z(X, (\Q/\Z)')$ is finite. 
\end{lem}
\begin{proof}
   % By Lemma \ref{lem:Semi-Purity}, we have  $H^2_Z(X, (\Q/\Z)') =0$ if $c\geq 2$. We can therefore assume that $c=1$. We need to show $H^2_Z(X, (\Q/\Z)')$ is a finite group. 
    We first assume that $Z$ is also regular. Then by Gabber's purity theorem \cite{Fuji02}, we have an isomorphism $H^2_Z(X, (\Q/\Z)') \cong H^0(Z, (\Q/\Z)'(-1))$. But then this group is finite by Lemma \ref{lem:et-como-finite*-1} because $Z$ is regular. 

    We now prove the lemma for general closed subset $Z \subset X$. Let $W = Z_{{\rm sing}}$. Then ${{\rm codim}}(W, Z) \geq 1$ and hence ${{\rm codim}}(W, X)\geq 2$. By Lemma \ref{lem:Semi-Purity}, we therefore have  $H^2_W(X, (\Q/\Z)')=0$. We now use the long exact sequence  
    \begin{equation}\label{eqn:et-supp-*2.1}
         H^2_{W}(X,  (\Q/\Z)' ) \to  H^2_Z(X, (\Q/\Z)' ) \to  
         H^2_{Z\setminus W}(X\setminus W, (\Q/\Z)'  ) \to \cdots
    \end{equation}
    for the triple $(X, Z, W)$ to conclude that  
    $H^2_Z(X, (\Q/\Z)' ) \inj 
         H^2_{Z\setminus W}(X\setminus W, (\Q/\Z)'  )$. It therefore suffices to show that the later group $H^2_{Z\setminus W}(X\setminus W, (\Q/\Z)' )$ is finite. But this follows from the above paragraph because $X \setminus W$ and $Z \setminus W$  are regular. This completes the proof. 
\end{proof}

\section{Prime to $p$-primary torsion} \label{sec:Prime-to-p}

In this section, we let $k$ denote a local field of exponential characteristic $p \geq 1$ and let $X$ be a regular quasi-projective scheme over $k$. We study the prime to $p$-primary torsion subgroup
$\pi^{\ab}_1(X)\{p'\}$. Note that if $k$ is a local field of characteristic zero, then $p=1$ and $\pi^{\ab}_1(X)\{p'\} = \pi^{\ab}_1(X)_{{\rm tor}}$. For a set $\mathbb{M}$ of prime integers, we let 
$\pi^{\ab}_1(X)_{\mathbb{M}} = \varprojlim_{n} \pi^{\ab}_1(X)/n$, where the limit is taken over all integers $n$ whose prime divisors lie in $\mathbb{M}$. 
%In the case when $\mathbb{M} = \lbrace p \rbrace$ for $p \geq 2$, we write $\pi^{\ab}_1(X)_{\mathbb{M}}$ as $\pi^{\ab}_1(X)_{p}$. 
Recall from \cite[Lemma 3.7]{TCFT-arXiv} that 
\begin{equation} \label{eqn:prime-to-p-*1}
    \pi^{\ab}_1(X) = \pi^{\ab}_1(X)_{\mathbb{P}} \oplus \pi^{\ab}_1(X)_{p},
\end{equation}
where $\mathbb{P}$ is the set of all primes different from $p$ and 
$ \pi^{\ab}_1(X)_{p} =  \pi^{\ab}_1(X)_{\lbrace p \rbrace}$ when $p \geq 2$ and zero when $p=1$. Moreover, the Pontryagin dual $(\pi^{\ab}_1(X)_{\mathbb{P}})^{\vee}$ 
%of ${\pi^{\ab}_1(X)}_{\mathbb{P}}$ 
is isomorphic to 
$H^1(X, (\Q/\Z)')$ and 
\begin{equation} \label{eqn:prime-to-p-*2}
     \pi^{\ab}_1(X)\{p'\} =\pi^{\ab}_1(X)_{\mathbb{P}}\{p'\} =(\pi^{\ab}_1(X)_{\mathbb{P}})_{{\rm tor}}. 
\end{equation}

We now prove the main result that helps us to study the prime to $p$-primary torsion
subgroup of $\pi^{\ab}_1(X)$. 

\begin{lem}\label{lem:AEFG-X-U}
    Let $X$ be a regular quasi-projective scheme over a local field $k$. Let $U \inj X$ be a dense open subset of $X$. Then 
    we have the following short exact sequence. 
    \begin{equation}\label{eqn:AEFG-X-U*4}
    0 \to F \to \pi^{\ab}_1(U)\{p'\}\to \pi^{\ab}_1(X)\{p'\} \to 0,
\end{equation} 
 where $F$ is a finite group. In particular, $\pi^{\ab}_1(X)\{p'\}$ is finite if and only if $\pi^{\ab}_1(U)\{p'\}$ is finite. 
\end{lem}
\begin{proof}
    The long exact sequence for the \'etale cohomology groups for the pair $(X, U)$ is as follows. 
    \begin{equation} \label{eqn:AEFG-X-U*1}
        H^1_Z(X, (\Q/\Z)') \to  H^1(X, (\Q/\Z)') \to  H^1(U, (\Q/\Z)') \to  H^2_Z(X, (\Q/\Z)') \to \cdots 
    \end{equation}
where $Z = X \setminus U$. Since $U$ is dense in $X$, ${{\rm codim}}(Z, X) \geq 1$. By semi-purity Lemma \ref{lem:Semi-Purity}, we therefore have $H^1_Z(X, (\Q/\Z)')=0$ and by Lemma \ref{lem:et-coho-supp-fin-*2}, the group $H^2_Z(X, (\Q/\Z)') $ is a finite group. The exact sequence \eqref{eqn:AEFG-X-U*1} then yields the following short exact sequence. 
\begin{equation}\label{eqn:AEFG-X-U*2}
    0 \to H^1(X, (\Q/\Z)') \to  H^1(U, (\Q/\Z)') \to F' \to 0,
\end{equation}
where $F'$ is a finite group. 
Taking Pontryagin dual, we get the following short exact sequence. 
\begin{equation}\label{eqn:AEFG-X-U*3}
    0 \to F \to \pi^{\ab}_1(U)_{\mathbb{P}} \to \pi^{\ab}_1(X)_{\mathbb{P}} \to 0,
\end{equation}
where $F$ is a finite group. Taking torsion subgroups therefore yields the short exact sequence 
\begin{equation}\label{eqn:AEFG-X-U*4.5}
    0 \to F \to (\pi^{\ab}_1(U)_{\mathbb{P}})_{{\rm tor}} \to (\pi^{\ab}_1(X)_{\mathbb{P}})_{{\rm tor}} \to 0.
\end{equation}
The lemma follows from \eqref{eqn:AEFG-X-U*4.5} and \eqref{eqn:prime-to-p-*2}.
\end{proof}

We now prove the finiteness of prime to $p$-primary torsion subgroup of the abelian \'etale fundamental group of a regular quasi-projective scheme over a local field. The following result was proven in \cite[Theorem~12.8]{TCFT-arXiv} when $X$ is a smooth quasi-projective scheme over a local field $k$ that admits a smooth compactification.

\begin{thm}\label{thm:A.1}
    Let $X$ be a regular quasi-projective scheme over a local field $k$. Then $\pi^{\ab}_1(X)\{p'\}$ is finite. 
\end{thm}
\begin{proof}
     Without loss of generality, we assume that $X$ is a regular quasi-projective integral scheme over $k$ and let 
    $X \inj \overline{X}$ be a compactification of $X$.  Let $\phi \colon W \to \overline{X}$ be an alteration morphism. 
    As in the proof of Lemma~\ref{lem:et-como-finite*-1}, there exists  a dense open subset $U \subset X$ 
    such that the induced morphism $\phi_{U'} \colon U'=\phi^{-1}(U)\to U$ is a finite flat morphism of degree $e$,  where $e =1$ if ${{\rm Char}}(k)=0$ and $e=p^r$ if ${{\rm Char}}(k)=p>0$.
    By \cite[XVII, Proposition~6.2.5]{SGA4}, we have a map 
    $(\phi_{U'})_*\colon H^1(U', (\Q/\Z)') \to H^1(U, (\Q/\Z)')$ such that the composition 
    \begin{equation} \label{eqn:A.1.0}
        H^1(U, (\Q/\Z)') \xrightarrow{\phi_{U'}^*} H^1(U', (\Q/\Z)') \xrightarrow{(\phi_{U'})_*} H^1(U, (\Q/\Z)')
    \end{equation}
    is the multiplication by $e$. Note that all groups in \eqref{eqn:A.1.0} are prime to $p$-torsion groups. Since $e$ is either $1$ or $p^r$, it follows that the composition of the arrows in \eqref{eqn:A.1.0} is an isomorphism. Taking Pontryagin dual, we get a group homomorphism $\phi_{U'}^* \colon \pi_1^{\ab}(U)_{\mathbb{P}} \to\pi_1^{\ab}(U')_{\mathbb{P}} $ such that the composite map
    \begin{equation} \label{eqn:A.1.1}
         \pi_1^{\ab}(U)_{\mathbb{P}} \xrightarrow{\phi_{U'}^*}  \pi_1^{\ab}(U')_{\mathbb{P}} \xrightarrow{(\phi_{U'})_*}  \pi_1^{\ab}(U)_{\mathbb{P}}
    \end{equation}
is an isomorphism.
%since it is the Pontyagin dual of the isomorphism in \eqref{eqn:A.1.0}. 
Taking torsion subgroups and using \eqref{eqn:prime-to-p-*2}, we get that the right vertical arrow in the following diagram is injective. 
\begin{equation}\label{eqn:A.1.2}
    \xymatrix@C.8pc{
  0 \ar[r]  &  F_1 \ar[r] & \pi^{\ab}_1(X)\{p'\} \ar[r] & \pi^{\ab}_1(U)\{p'\} \ar[r] \ar@{^{(}->}[d]^{\phi_{U'}^*} &   0\\
0\ar[r] &  F_2 \ar[r] & \pi^{\ab}_1(W)\{p'\} \ar[r] & \pi^{\ab}_1(U')\{p'\} \ar[r] &   0,
    }
\end{equation}    
where $F_i$ are finite groups and the horizontal rows are exact sequences by \eqref{eqn:AEFG-X-U*4}. By a diagram chase, it suffices to prove that $\pi^{\ab}_1(W)\{p'\}$ is finite. But this follows from \cite[Theorem~12.8]{TCFT-arXiv} as $W$ is a smooth projective scheme over a finite extension $k'$ of $k$. 
\end{proof}

\section{$p$-primary torsion}\label{sec:p-part}

In this section, we let $X$ denote a geometrically integral quasi-projective regular scheme of dimension $d \geq 0$ over a local field $k$ of positive characteristic $p>0$. 
Recall that we let $\pi_1^{\ab}(X)^{{\rm geom}}:= {\Ker}(\pi_1^{\ab}(X) \to \pi_1^{\ab}(k))$.
We prove a (weak) structure theorem for $\pi_1^{\ab}(X)^{{\rm geom}}$ and use it to conclude that $p$-primary torsion $\pi_1^{\ab}(X)\{p\}$ is finite. 
We start with the following general result which is a corollary of the Galois theory. For the sake of completion, we include its proof. 

\begin{lem}\label{lem:coker-Galois}
    Let $L/K$ be a finite extension of fields. Then the induced map $G_L \to G_K$ has a finite cokernel, i.e., the image of $G_L$ has finite index in $G_K$.
\end{lem}
\begin{proof}
    Let $K \subset F \subset L$ be the separable closure of $K$ in $L$.     
    Let $ F^{{\rm sep}}$ be a separable closure of $F$. Note that the field extension $F^{{\rm sep}} L / L $ is a separable algebraic extension. We let $L^{{\rm sep}}$ be a separable closure of $L$ such that $L F^{{\rm sep}} \subset L^{{\rm sep}}$. Since $F^{{\rm sep}}$ is separably closed, the extension $L^{{\rm sep}} / F^{{\rm sep}}$ is purely inseparable, and hence $L^{{\rm sep}} / L F^{{\rm sep}}$ is also purely inseparable. But since $L^{{\rm sep}}/L$ is separable, it follows that $L^{{\rm sep}}/ L F^{{\rm sep}}$ is also separable. This implies that $L^{{\rm sep}} = L F^{{\rm sep}}$. Since $L/F$ is purely inseparable, a similar argument yields that $F^{{\rm sep}} \cap L = F$.
    In particular, the restriction map 
    $G_L = {{\rm Gal}}(L^{{\rm sep}}/ L) = {{\rm Gal}}(L F^{{\rm sep}}/L) \xrightarrow{\cong} {{\rm Gal}}(F^{{\rm sep}}/F) = G_F$ is an isomorphism. Since $F/K$ is finite separable, $F^{{\rm sep}} = K^{{\rm sep}}$ and the restriction map  
    $G_L = {{\rm Gal}}(L^{{\rm sep}}/ L) \to G_K = {{\rm Gal}}(K^{{\rm sep}}/ K)$ therefore factors as 
      $G_L \xrightarrow{\cong} G_F \to G_K$. It now suffices to show that the image of the map $G_F \to G_K$
      has a finite index. But this follows because $G_K$ acts on the finite set ${{\rm End}}_K(F, K^{{\rm sep}})$ of $K$-embeddings of $F$ in $K^{{\rm sep}}$ transversely and the stabilizer of the inclusion $F \subset K^{{\rm sep}}$ is the subgroup $G_F \inj G_K$. Note that we have proven that 
      $|G_K/G_L| \leq |{{\rm End}}_K(F, K^{{\rm sep}})| = [F:K] = [L:K]_{{\rm sep}} \leq [L:K]$. 
\end{proof}

\begin{lem} \label{lem:coker-pi-ab}
    Let $\phi\colon W \to Y$ be a generically finite dominant morphism between integral normal schemes. Then the induced map $\phi_* \colon \pi_1^{\ab}(W) \to \pi_1^{\ab}(Y)$ has finite cokernel. 
\end{lem}
\begin{proof}
    Let $L$ and $K$ denote the field of fractions of $W$ and $Y$, respectively. Then since $\phi$ is a generically finite dominant morphism, it induces a finite morphism $\phi \colon \Spec(L) \to \Spec(K)$. We consider the following commutative diagram. 
    \begin{equation} \label{eqn:coker-pi-ab*1}
        \xymatrix@C.8pc{
        \pi_1^{\ab}(L) \ar@{->>}[r] \ar^{\phi_*}[d] & \pi_1^{\ab}(W) \ar^{\phi_*}[d]\\
         \pi_1^{\ab}(K) \ar@{->>}[r] & \pi_1^{\ab}(Y). 
        }
    \end{equation}
    Since $W$ and $Y$ are normal, the horizontal arrows are surjective (see \cite[Expos\'e V, Proposition~8.2]{SGA1}). In particular, it suffices to show that the left vertical arrow in \eqref{eqn:coker-pi-ab*1} has a finite cokernel.  Since $\pi_1^{\ab}(L)$ (resp. $\pi_1^{\ab}(K)$) is the topological abelianization of $G_L$ (resp. $G_K$), it suffices to show that the image of $G_L$ has finite index in $G_K$ (under the restriction map $G_L \to G_K$). But this follows from Lemma \ref{lem:coker-Galois} because by assumption on the morphism $\phi$, the field extension $L/K$ is a finite extension of fields. 
\end{proof}

Since we want to apply Lemma \ref{lem:coker-pi-ab} for the morphism induced by a base change to a separable closure, we need the following lemma. 

\begin{lem} \label{lem:Base-change-Gen-fin}
  Let $W$ and $Y$ be normal geometrically connected schemes over fields $l_1$ and $l_2$, respectively. Assume that $l_2 \subset l_1$ and $l_1/l_2$ is a finite extension. Let $\phi \colon W \to Y$ be a generically finite dominant morphism such that the following diagram commutes. 
  \begin{equation} \label{BC-GF*-1}
      \xymatrix@C.8pc
      {
      W \ar[r]^-{\phi} \ar[d] & Y \ar[d] \\
      \Spec(l_1) \ar[r] & \Spec(l_2).}
  \end{equation}
 For $i=1, 2$, let $l_i^{{\rm sep}}$ denote a separable closure of the field $l_i$ such that 
 $l_2^{{\rm sep}} \subset l_1^{{\rm sep}}$. Then the induced morphism $\phi' \colon W \times_{l_1} l_1^{{\rm sep}} \to Y \times_{l_2} l_2^{{\rm sep}}$  is
  a generically finite dominant morphism between integral normal schemes. 
\end{lem}
\begin{proof}
    Since $W$ and $Y$ are geometrically connected, it follows that $W \times_{l_1} l_1^{{\rm sep}}$ and $Y \times_{l_2} l_2^{{\rm sep}}$  are connected schemes over $l_1^{{\rm sep}}$ and $l_2^{{\rm sep}}$, respectively. Since $l_i^{{\rm sep}}$ is a separable extension of $l_i$, it follows that 
    $W \times_{l_1} l_1^{{\rm sep}}$ and $Y \times_{l_2} l_2^{{\rm sep}}$ are normal
    (see \cite[Lemma 0C3M]{StackP}). In particular, they are integral as well. We are reduced to show that the induced morphism $\phi' \colon W \times_{l_1} l_1^{{\rm sep}} \to Y \times_{l_2} l_2^{{\rm sep}}$ is
  a generically finite dominant morphism. To see this, we note that the morphism $\phi'$ is the following composition. 
  \begin{equation}\label{BC-GF*-2}
  \xymatrix@C.8pc{
       W \times_{l_1} l_1^{{\rm sep}} \ar[d]^-{\phi_1 \times_{l_1}  l_1^{{\rm sep}}} \\
       (Y \times_{l_2} l_1) \times_{l_1}  l_1^{{\rm sep}} \ar[r]^-{\cong}&
       Y  \times_{l_2}  l_1^{{\rm sep}} \ar[d]^-{{\rm id}_Y \times_{l_2} \psi} \\
       %(Y \times_{l_2} l_2^{{\rm sep}}) \times_{l_2^{{\rm sep}}}  l_1^{{\rm sep}} \ar[d]\\
       &  Y \times_{l_2} l_2^{{\rm sep}},}
  \end{equation}
    where $\psi\colon \Spec(l_2^{{\rm sep}}) \to \Spec(l_1^{{\rm sep}})$ is the induced map between separable closures and $\phi_1 \colon W \to Y \times_{l_2} l_1$ is the induced morphism such that the composition $ W \xrightarrow{\phi_1} Y \times_{l_2} l_1 \to Y$ is the given morphism $\phi$. As we have seen in the proof of Lemma \ref{lem:coker-Galois}, the field extension $l_2^{{\rm sep}}/ l_1^{{\rm sep}}$ is a finite extension. In particular, the morphism 
    ${{\rm Id}}_Y  \times_{l_2} \psi \colon Y  \times_{l_2}  l_1^{{\rm sep}} 
    \to 
       Y \times_{l_2} l_2^{{\rm sep}} $ is a finite surjective morphism. Since  the composition of 
      generically finite dominant morphisms are again generically finite dominant, it suffices to show that $\phi_1 \colon W \to Y \times_{l_2} l_1$ is a generically finite dominant morphism. 
    
    Since the morphism $\phi$ is generically finite morphism, it follows that so is $\phi_1$ (see \cite[Exercise II.4.8]{Hart}). We are now left to show that $\phi_1$ is dominant. To see this, we let $l_2 \subset l' \subset l_1$ be the separable closure of $l_2$ in $l_1$. Since $l'/l_2$ is separable, it follows that  $Y \times_{l_2} l'$ is  normal (see \cite[Lemma~0C3M]{StackP}). Since $Y$ is geometrically connected, $Y \times_{l_2} l'$ is also connected. We conclude that $Y \times_{l_2} l'$ is connected normal scheme and hence irreducible. Since $l_1/l'$ is purely inseparable, the map $\Spec(l_{1}) \rightarrow \Spec(l')$ is an universal homeomorphism (see \cite[Section~0BR5]{StackP}), and hence the map
    $Y \times_{l_2} l_1 = (Y \times_{l_2} l') \times_{l'} l_1  \rightarrow Y \times_{l_{2}} l'  $ is a homeomorphism. We conclude that $Y \times_{l_2} l_1$ is also irreducible. Since the projection morphism $Y \times_{l_2} l_1 \to Y$ is a finite morphism, $\phi \colon W \to Y$ is dominant and $Y \times_{l_2} l_1$ is irreducible, it follows that $\phi_1\colon W \to Y \times_{l_2} l_1$ is a dominant morphism. This completes the proof. 
\end{proof}

We now prove a structure theorem for $\pi_1^{\ab}(X)^{{\rm geom}}$. 
\begin{prop} \label{prop:Struc-geom-*1}
   Let  $X$ denote a geometrically connected and normal projective scheme of dimension $d \geq 0$ over a local field $k$ of exponential characteristic $p\geq 1$. Then we have the following exact sequence. 
   \begin{equation} \label{eqn:Struc-geom-*1.0}
       \widehat{\Z}^r \to \pi_1^{\ab}(X)^{{\rm geom}} \to F \to 0,
   \end{equation}
   where $F$ is a finite group and $r \geq 0$. In particular, $\pi_1^{\ab}(X)^{{\rm geom}}$ is a topologically finitely generated profinite abelian group.  Moreover, if $\pi_1^{\ab}(X)^{{\rm geom}} \surj H$ is a topological quotient of $\pi_1^{\ab}(X)^{{\rm geom}}$ and $l$ is a prime (including $p$), then $l$-primary torsion subgroup $H\{l\}$ is finite. 
\end{prop}
\begin{proof}
   Let us first assume the existence of the exact sequence \eqref{eqn:Struc-geom-*1.0}. Then the profinite abelian group $\pi_1^{\ab}(X)^{{\rm geom}}$, and hence its topological quotient, is a finitely generated profinite abelian group. The last claim of the lemma then follows from \cite[Theorem~4.3.5]{Pro-fin}. Indeed, the topological quotient $H$ is a finitely generated profinite abelian group with $d(H) \leq r$ and the loc. cit. yields that for each prime $l$, the $l$-primary torsion subgroup of $H$ is finite. We now prove the exact sequence \eqref{eqn:Struc-geom-*1.0}.

    Let $\phi \colon W \to X$ be an alteration (resp. a desingularization) of $X$ of de Jong and Gabber \cite{Tem17} when ${\rm Char}(k)=p>0$ (resp. when ${\rm Char}(k)=0$). Since a desingularization map is an alteration of degree $1$, we work with alteration $\phi\colon W \to X$ irrespective of the characteristic of $k$. 
    Note that $\phi$ is a generically finite morphism between integral schemes over $\Spec(k)$. Moreover, there exists a finite field extension $k'/k$ such that $W \to \Spec(k')$ is a smooth morphism.  Note that the connected components of $W \otimes_{k'} \overline{k'}$ are obtained as the base change of the connected components of $W \otimes_{k'} k'^{sep}$ via the map $\Spec(\overline{k'}) \rightarrow \Spec(k'^{sep})$, since the latter map is an universal homeomorphism. In particular, the connected components of $W \otimes_{k'} \overline{k'}$  are  defined over a finite separable extension $k''/k'$ and are geometrically connected. 
    In other words, there exists a finite extension $k''/k'$ such that connected components of $W \otimes_{k'} k''$ are geometrically connected. Let $W_1 \to \Spec(k'')$ be such a component. Since the induced morphism  $W_1 \to W$ is \'etale morphism, the composite morphism $W_1 \to X$ is also an alteration such that $W_1 \to \Spec(k')$ is geometrically connected and smooth. We therefore assume that $W$ is geometrically connected and smooth over $k'$.

    We let $k \subset k^{{\rm sep}}$ (resp. $k' \subset k'^{{\rm sep}}$) be a separable closure of $k$ (resp. of $k'$). Since $W/k'$ and $X/k$ are geometrically connected
    \cite[(1.7)]{KL81} (also see \cite[IX, Theorem~6.1]{SGA1}) yields that the rows in the following diagram are exact sequences. 
    \begin{equation} \label{eqn:Struc-geom-*1.1}
        \xymatrix@C.8pc{
        \pi^{\ab}_1(W\otimes_{k'} k'^{{\rm sep}}) \ar[r] \ar[d]^-{\alpha} & \pi^{\ab}_1(W) \ar[r]^-{\theta'} \ar[d]^-{\beta} & \pi^{\ab}_1(k') \ar[r] \ar[d]^-{\gamma}& 0\\
         \pi^{\ab}_1(X\otimes_{k} k^{{\rm sep}}) \ar[r] & \pi^{\ab}_1(X) \ar[r]^-{\theta} & \pi^{\ab}_1(k) \ar[r]& 0. 
        }
    \end{equation}
    Since all arrows in the above diagram are natural push-forward maps, both the squares in 
    \eqref{eqn:Struc-geom-*1.1} are commutative. Note that in the loc. cit., they base change to the algebraic closures but it does not matter as the \'etale fundamental groups are invariant under topological homeomorphisms. 
    In particular, we get surjections $ \pi^{\ab}_1(W\otimes_{k'} k'^{{\rm sep}}) \surj \Ker(\theta')$ and $ \pi^{\ab}_1(X\otimes_{k} k^{{\rm sep}}) \surj \Ker(\theta)$.  Since $W \to \Spec(k')$ is geometrically connected and smooth, it follows from \cite[Theorem~1.1]{Yos03} that 
    $\Ker(\theta') = \widehat{\Z}^r \oplus F_1$, where $r\geq 0$ and $F_1$ is a finite group.
    If we let $\beta' \colon \Ker(\theta') \to \Ker(\theta) = \pi_1^{\ab}(X)^{{\rm geom}}$ denote the induced map, then we have the following exact sequences.  
    \begin{equation}\label{eqn:Struc-geom-*1.2}
        \widehat{\Z}^r \oplus F_1 \xrightarrow{\beta'}  \pi_1^{\ab}(X)^{{\rm geom}} \to {{\rm coker}}(\beta')
    \end{equation}
    and
    \begin{equation}\label{eqn:Struc-geom-*1.3}
        {{\rm coker}}(\alpha) \to {{\rm coker}}(\beta') \to 0. 
    \end{equation}
   We claim that ${{\rm coker}}(\beta')$ is finite.  By the exact sequence \eqref{eqn:Struc-geom-*1.3}, it suffices to show that ${{\rm coker}}(\alpha)$ is finite. But this follows from Lemmas \ref{lem:coker-pi-ab} and \ref{lem:Base-change-Gen-fin}. Indeed, by Lemma \ref{lem:Base-change-Gen-fin}, the induced morphism $W\otimes_{k'} k'^{{\rm sep}}  \to X\otimes_{k} k^{{\rm sep}}$ is a generically finite dominant morphism between integral normal schemes. But then by Lemma \ref{lem:coker-pi-ab}, the induced map $\alpha$ on abelian fundamental groups has finite cokernel. We conclude that ${{\rm coker}}(\beta')$ is finite. We now define the map $\beta'' := \beta' \circ i $, where $i : \widehat{\Z}^r \hookrightarrow  \widehat{\Z}^r \oplus F_{1} $ is the inclusion of the first factor. Since the maps $\beta'$ and $i$ have finite cokernels, the map $\beta''$ also has finite cokernel. Hence, we obtain an exact sequence
   \begin{equation}\label{eqn:Struc-geom-*1.2.1}
        \widehat{\Z}^r \xrightarrow{\beta''}  \pi_1^{\ab}(X)^{{\rm geom}} \to {{\rm coker}}(\beta'')\to 0,
    \end{equation}
where the group ${{\rm coker}}(\beta'')$ is finite. Setting $F = {{\rm coker}}(\beta'')$ completes the proof. 
\end{proof}

 \begin{thm} \label{thm:Main-thm-2} 
      Let  $X$ denote a geometrically connected projective regular scheme of dimension $d \geq 0$ over a local field $k$ of positive characteristic $p>0$. Then $p$-primary torsion subgroup  $\pi_1^{\ab}(X)\{p\}$ is finite. 
 \end{thm}
\begin{proof}
By taking $p$-primary torsion subgroups of the exact sequence 
$0 \to \pi_1^{\ab}(X)^{{\rm geom}} \to \pi^{\ab}_1(X) \to \pi^{\ab}_1(k) \to 0$, we get the following exact sequence. 
\begin{equation}\label{eqn:Main-thm-2*-1}
    0 \to \pi_1^{\ab}(X)^{{\rm geom}}\{p\} \to \pi^{\ab}_1(X)\{p\} \to \pi^{\ab}_1(k)\{p\}. 
\end{equation}
Since $k$ is a local field of positive characteristic $p>0$, it follows $\pi^{\ab}_1(k)\{p\} = G_k^{\ab}\{p\} = k^{\times} \{p\} = 0$. The lemma now follows from
Proposition \ref{prop:Struc-geom-*1}. 
%Indeed, by the proposition, $\pi_1^{\ab}(X)^{{\rm geom}}$ is a finitely generated pro-finite abelian group and hence has finite $p$-primary torsion by \cite[Theorem~4.3.5]{Pro-fin}. This completes the proof. 
\end{proof}

 \subsection{Proof of Theorem \ref{thm:A}} \label{sec:proof-thm:A}
Theorem \ref{thm:A} follows immediately from  Theorems \ref{thm:A.1} and \ref{thm:Main-thm-2}. 
 \qed

\section{Structure of $SK_1(X)$} \label{sec:SK_1(X)}

In this section, we let $f \colon X \to \Spec(k)$ denote an integral projective scheme of dimension $d \geq 0$ over a local field $k$. Recall that $SK_1(X)$ is defined as follows. 
\begin{equation}\label{eqn:def-SK_1*-1}
    SK_1(X) := {{\rm Coker}}(\oplus_{y\in X_{(1)}} K_2(k(y)) \xrightarrow{(\partial^y_x)} \oplus_{x \in X_{(0)}} k(x)^{\times}),
\end{equation}
where $\partial^y_x = 0$ if $x \not\in C_y := \overline{\{y\}}$ and it is the following composition map, otherwise. 
\begin{equation} \label{eqn:def-SK_1*-2}
    K_2(k(y)) \xrightarrow{\partial_A} \oplus_{i=1}^r k(x_i)^{\times} 
    \xrightarrow{\sum_i  {\rm Nm}_{k(x)}^{k(x_i)}} k(x)^{\times},
\end{equation}
where the semi-local ring $A$ is the normalization of $\mathcal{O}_{C_y, x}$, and $x_1 , \dots , x_r$ are its closed points. The first arrow in \eqref{eqn:def-SK_1*-2} is the boundary map for the localization sequence for the regular ring $A$ and the second arrow is the sum of the Norm maps. 
If $X$ is a curve, then $SK_1(X) = {{\rm Im}}(\oplus_{x \in X_{(0)}} k(x)^{\times} \to K_1(X)) \subset K_1(X)$. 

By a reciprocity law for $K$-groups, it follows that the norm maps ${\rm Nm}^{k(x)}_k$ for $x \in X_{(0)}$ induce a push-forward map $f_* \colon SK_1(X) \to k^{\times}$. We also denote this map by ${\rm Nm}_{X/k}$. 
We let $V(X/k)$ denote its kernel, i.e., 
\begin{equation}\label{eqn:def-SK_1*-4}
    V(X/k) := \Ker(f_* \colon SK_1(X) \to k^{\times}). 
\end{equation}
We know the structure of $k^{\times}$ for local fields, so to determine the structure of $SK_1(X)$, it suffices to study $V(X/k)$. Our main theorem about the structure of $V(X/k)$ is Theorem~\ref{thm:B'} which, for convenience,  we restate as follows.

\begin{thm} \label{thm:B}
    Let $X$ be a geometrically irreducible 
    (for example, normal geometrically connected) projective integral scheme of dimension $d \geq 0$ over a local field $k$. Then 
    \begin{equation} \label{eqn:thm-B*-1}
        V(X/k) = T \oplus D,
    \end{equation}
    where $T$ is a torsion group and $D$ is a divisible group. 
\end{thm}

We recall here that a normal integral scheme over a field is geometrically connected if and only if it is geometrically irreducible.  If $X$ is a geometrically connected smooth projective scheme over a local field, the above theorem follows from \cite[Proposition~2.3]{Sato05} 
(also see \cite[Theorem~1.3]{Yos03} and \cite[Proposition~2.1]{Forre-Crelle}). Note that if $d=0$, then $X$ is the spectrum of a local field $l$ containing $k$ such that $l/k$ is a finite purely inseparable field extension. In this case, the norm map $SK_1(X) = l^{\times} \to k^{\times}$ is the power $p^r$-map, where $p = {{\rm Char}}(k)$ and $r\geq 0$. In particular, $V(X/k) = {\Ker}( l^{\times} \to k^{\times})=0$. We can therefore assume that $d \geq 1$. 

\subsection{Reduction to the case of  curves} \label{sec:Red-to-curve}

In this section, we assume that Theorem \ref{thm:B} is known for geometrically connected projective curves over local fields and prove the theorem under this assumption. 

\begin{lem} \label{lem:red-to-curve}
    Assume that Theorem \ref{thm:B} holds for all geometrically irreducible projective integral curves. Then the theorem holds for all geometrically irreducible projective integral schemes over local fields.
\end{lem}
\begin{proof}
    Let $X$ be a geometrically irreducible projective integral scheme of dimension $d \geq 1$ over a local field $k$. If $d=1$, then the lemma for $X$ follows from its assumption. We can therefore assume that $d\geq 2$. 
    We let $\mathcal{C}_X$ denote the set of dimension one geometrically irreducible integral closed subschemes of $X$. For $C \in \mathcal{C}_X$,  Theorem \ref{thm:B} holds for
    $C \to \Spec(k)$ and 
    \begin{equation} \label{eqn:red-curve*-1}
        V(C/k) = D_C \oplus T_C,
    \end{equation} 
    where $D_C$ is a divisible group and $T_C$ is a torsion group. By Definition \eqref{eqn:def-SK_1*-1}, there exists a push forward map $(\iota_C)_* \colon SK_1(C) \to SK_1(X)$ for all $C \in \mathcal{C}_X$ such that the following diagram commutes. 
    \begin{equation}\label{eqn:red-curve*-2}
        \xymatrix@C2pc{
        SK_1(C) \ar[r]^-{{\rm Nm_{C/k}}} \ar[d]^-{(\iota_C)_*} & k^{\times} \ar@{=}[d]\\
        SK_1(X) \ar[r]^-{{\rm Nm_{X/k}}} & k^{\times}.
        }
    \end{equation}
    In particular, we have a well-defined map 
    \begin{equation}\label{eqn:red-curve*-3}
       \Phi\colon  \oplus_{C \in\mathcal{C}_X} V(C/k) \to V(X/k).
    \end{equation}
    We let $D \subset V(X/k)$ denote the image $\Phi(\oplus_{C \in \mathcal{C}_X} D_C)$ of the direct sum of the divisible part of $V(C/k)$ in \eqref{eqn:red-curve*-1}. Since the direct sum and the image of divisible groups are again divisible, the subgroup $D$ is a divisible group. It therefore suffices to show that the quotient group  $V(X/k)/D$ is a torsion group. 

    Let $\alpha \in V(X/k)$. By Definition \eqref{eqn:def-SK_1*-1} of $SK_1(X)$, there exist finitely many closed points $x_1, \dots, x_r \in X_{(0)}$ and $a_i \in k(x_i)^{\times}$ such that $\alpha = \sum_i a_i$ and $\prod_i {\rm Nm}_{k}^{k(x_i)} (a_i) =1$. Since $d \geq 2$, $X$ is geometrically irreducible and ${\rm codim}(\{x_1, \dots, x_r\},X) = d$, it follows from the claim about the geometrically irreducible part of \cite[Theorem~1]{AK} 
    %(proven in the last para on Page~779 of loc. cit.) 
    that the intersection of $(d-1)$-generic hypersurfaces on $X$ containing $\{x_1, \dots, x_r\}$ is geometrically irreducible. 
   Let $C_0'$ be such an intersection and let 
   $C_0 = (C_0')_{{\rm red}}$. By \cite[Lemma~054M]{StackP}, $C_0 \subset X$ is  
 %   In particular,  there exists
    a geometrically irreducible integral curve  such that $\{x_1, \dots, x_r\} \subset C_0$. 
    
    It follows from \eqref{eqn:def-SK_1*-1} for $SK_1(C_0)$ and \eqref{eqn:red-curve*-2} that there exists $\beta = \sum_i a_i \in V(C_0/k)$ such that 
    $\Phi(\beta) = \alpha$. Writing $\beta = \beta_1 + \beta_2$ such that $\beta_1 \in D_{C_{0}}$ and $\beta_2 \in T_{C_{0}}$, we have 
    \begin{equation}
        \overline{\alpha} = \overline{\Phi(\beta_1 + \beta_2)} = \overline{\Phi(\beta_2)} \in V(X/k)/D.
    \end{equation}
    But then the above element is a torsion element as $T_{C_{0}}$ is a torsion group. This proves the claim that  $V(X/k)/D$ is a torsion group and hence completes the proof of the lemma. 
\end{proof}

\subsection{Theorem \ref{thm:B} for geometrically irreducible curves}

In this section, we prove Theorem \ref{thm:B} for geometrically irreducible curves over local fields, and hence by the virtue of Lemma \ref{lem:red-to-curve}, we complete the proof of Theorem \ref{thm:B}. We start by proving that the normalization of a geometrically irreducible scheme is again geometrically irreducible. Recall that a morphism $W \to \Spec(l)$ is said to be geometrically irreducible over $l$ if its base change $W \times_l {\overline{l}}$ is irreducible, where $\overline{l}$ is an algebraic closure of $l$. In the above case, we say $W$ is geometrically irreducible over $l$ (and we suppress morphism from the notations).

\begin{lem} \label{lem:geom-irr-normalization}
    Let $f \colon W \to \Spec(l)$ be an integral geometrically irreducible scheme over a field $l$. 
    Let $\phi \colon W^N \to W$ be the normalization of $W$. Then $W^N$ is also geometrically irreducible over $l$. 
\end{lem}
\begin{proof}
    The lemma follows because the normalization morphism is a birational morphism 
    for integral schemes and being geometrically irreducible is a birational invariant (see \cite[Lemma~054Q]{StackP}). 
\end{proof}

Let $k$ be a local field and let $f \colon C \to \Spec(k)$ be a geometrically irreducible integral projective curve. We let $\phi \colon C^N \to C$ be the normalization of $C$ and let $g = f \circ \phi \colon C^N \to \Spec(k)$ be the composition map. By Lemma \ref{lem:geom-irr-normalization}, it follows that $C^N$ is a geometrically irreducible (and therefore geometrically connected regular) scheme over $k$. By \cite[Lemma 0BY4]{StackP}, there exists a finite purely inseparable field extension $k'/k$ such that the normalization $C'$ of the integral curve $(C^N \otimes_k k')_{{\rm red}}$ is geometrically irreducible smooth projective curve over $k'$. 
We let $\phi' \colon C' \to C^N$ be the induced morphism. Being the composition of finite morphisms, $\phi'$ is finite. 
%(see the discussion preceding Lemma \ref{lem:coker-alt-V}). 
We then have the following commutative diagram. 
\begin{equation} \label{eqn:dia-curves}
    \xymatrix@C.8pc
    {C' \ar[r]^-{f'} \ar[d]^-{\phi'}  & \Spec(k') \ar[d]\\
   % C^N_l \ar[r]^-{g_l}  \ar[d]^-{\phi_1}& \Spec(l)\ar[d]\\
    C^N \ar[r]^-{g} \ar[d]^{\phi} & \Spec(k)\ar@{=}[d] \\
    C \ar[r]^-{f} & \Spec(k). }
\end{equation}
By the definition of $SK_1(-)$, we have the following commutative diagram such that the rows are exact sequences, and the middle and the right vertical arrows are the induced push-forward maps.
\begin{equation} \label{eqn:dia-curves*-1}
    \xymatrix@C.8pc{
   0 \ar[r] &  V(C'/k') \ar[d]^-{\psi'} \ar[r] & SK_1(C') \ar[r]^-{f'_*}  \ar[d]^-{\phi'_*} & (k')^{\times} \ar[d]^-{{\rm Nm}^{k'}_k }\\
 %  0 \ar[r] &  V(C^N_l/l) \ar[r] \ar[d]^-{\psi_1} & SK_1(C^N_l) \ar[r]^-{(g_l)_*}  \ar[d]^-{(\phi_1)_*}& l^{\times} \ar[d]^-{{\rm Nm}^{l}_{k}}\\
   0 \ar[r] &  V(C^N/k) \ar[r] \ar[d]^-{\psi} & SK_1(C^N) \ar[r]^-{g_*} \ar[d]^{\phi_*} & k^{\times} \ar@{=}[d] \\
    0 \ar[r] &  V(C/k) \ar[r] & SK_1(C) \ar[r]^-{f_*} & k^{\times}. }
\end{equation}
Since $C'$ is a geometrically connected smooth projective curve over $k'$, we know from \cite[Theorem~1.3]{Yos03} that $V(C'/k') = D' \oplus T'$, where $D'$ is a divisible group and $T'$ is a torsion group (in fact, it can be assumed to be finite). We therefore have to understand the structure of ${\rm coker}(\psi \circ \psi')$. We do it step-wise and understand the cokernel of each of the maps (in the left vertical column in \eqref{eqn:dia-curves*-1}). 

\begin{lem} \label{lem:coker-nor-V}
    With notations as in \eqref{eqn:dia-curves*-1}, the group 
    ${\rm coker}(\psi)$ is a torsion group of finite exponent.  
\end{lem}
\begin{proof}
    By the bottom commutative diagram in \eqref{eqn:dia-curves*-1}, we have ${\rm coker}(\psi)\inj {\rm coker}(\phi_*)$. Since $C$ is integral, the normalization map $\phi \colon C^N \to C$ is a birational surjective map. Let $E \subset C$ be the finite subset such that the induced morphism $\phi \colon \phi^{-1}(C \setminus E) \to C\setminus E$ is an isomorphism. If $E= \emptyset$, then $\phi$ (and hence $\phi_*$) is an isomorphism and the cokernel of $\phi_*$ is therefore zero. We now assume that $E = \{y_1, \dots, y_r\}$ and choose a point $x_i \in \phi^{-1}(y_i)= E_i$ for each $i=1, \dots, r$. Then $k(x_i)/k(y_i)$ is a finite field extension of degree $n_i$. 
    We consider the following commutative diagram. 
    \begin{equation} \label{eqn:coker-nor-V*-1}
        \xymatrix@C.8pc{
      (\oplus_{x \in (C^N \setminus \phi^{-1}(E))} k(x)^{\times})  \oplus 
      (\oplus_{1\leq i \leq r} \oplus_{x \in  E_i } k(x)^{\times})
        \ar@{->>}[r] \ar[d] & SK_1(C^N) \ar[d]^-{\phi_*} \\
        (\oplus_{y \in (C \setminus E)} k(y)^{\times}) \oplus (\oplus_{1\leq i \leq r} k(y_i)^{\times}) 
        \ar@{->>}[r] & SK_1(C),
        }
    \end{equation}
    where the left vertical arrow is given by various norm maps. 
    Since the horizontal arrows are surjective, it suffices to show that the cokernel of the left vertical arrow is a torsion group of finite exponent. Since $E$ is finite and $\phi$ is an isomorphism on $C \setminus E$, it suffices to show that the cokernel of the map $\oplus_{x \in E_i} k(x)^{\times} \to k(y_i)^{\times}$ is a torsion group of finite exponent for each $i = 1, \dots, r$. But this follows because the composition of the following maps is the same as the multiplication by $n_i$.
    \begin{equation}  \label{eqn:coker-nor-V*-2}
       k(x_i)^{\times} \inj \oplus_{x \in E_i} k(x)^{\times} \xrightarrow{\sum_{x\in E_i} {\rm Nm}^{k(x)}_{k(y_i)}} k(y_i)^{\times}. 
    \end{equation}
    This completes the proof.  
\end{proof}

We now study ${\rm coker}(\psi')$. We note that $\phi'$ is a dominant and finite morphism between regular integral curves. Since any dominant map to a regular integral curve is flat (see \cite[III, Proposition 9.7]{Hart}), it follows that $\phi'$ is also flat. 
%But then it is a quasi-finite as well (because the dimension of the fibers doesn't change for a flat morphism). Since $\phi'$ is also proper, 
We conclude that $\phi'$ is a finite flat morphism between integral regular curves.

\begin{lem}\label{lem:coker-alt-V}
   With notations as in \eqref{eqn:dia-curves*-1}, the group 
    ${\rm coker}(\psi')$ is a torsion group of finite exponent.
\end{lem}
\begin{proof}
    We first note that $k'/k$ is a purely inseparable extension. We let $[k':k] = p^j$, where $p \geq 1$ is the exponential characteristic of $k$. Recall that in this case the norm map ${\rm Nm}^{k'}_{k} \colon (k')^{\times}\to k^{\times}$ is the power $p^j$-map and it is therefore injective.  By the  top commutative diagram in \eqref{eqn:dia-curves*-1}, we have 
    ${\rm coker}(\psi')\inj {\rm coker}(\phi'_*)$. As shown above, $\phi' \colon C' \to C^N$ is a finite flat surjective morphism between regular integral curves. 
    Let $n$ denote the degree of the map $\phi'$. We now consider the following commutative diagram. 
     \begin{equation} \label{eqn:coker-alt-V*-1}
        \xymatrix@C.8pc{
       \oplus_{z \in  (C')_{(0)}} k(z)^{\times} 
        \ar@{->>}[r] \ar[d] & SK_1(C') \ar[d]^-{\phi'_*} \\
        \oplus_{x \in (C^N)_{(0)}} k(x)^{\times}  
        \ar@{->>}[r] & SK_1(C^N). 
        }
    \end{equation}
    By the above diagram \eqref{eqn:coker-alt-V*-1}, it suffices to show that the cokernel of the left vertical arrow is a torsion group of finite exponent. In other words, it suffices to show that for all $x \in (C^N)_{(0)}$, the cokernel of the induced map $\oplus_{z \in (\phi')^{-1}(x)} k(z)^{\times} \to k(x)^{\times}$ is a torsion group of a fixed exponent $m$ (that does not depend on $x$). We claim it is a torsion group of exponent $ n!$. To see this, we observe that for a $z_0\in (\phi')^{-1}(x)$, the composition 
    \begin{equation}  \label{eqn:coker-alt-V*-2}
        k(x)^{\times} \to k(z_0)^{\times} \inj \oplus_{z \in (\phi')^{-1}(x)} k(z)^{\times} \xrightarrow{\sum {\rm Nm}^{k(z)}_{k(x)}} k(x)^{\times}
    \end{equation}
    is same as the composition  $ k(x)^{\times} \to k(z_0)^{\times} \xrightarrow{{\rm Nm}^{k(z_{0})}_{k(x)}} k(x)^{\times} $. Since the composition of the latter map is  multiplication by the degree $[k(z_0):k(x)]$ (see \cite[III, Corollary~7.5.3]{Wei-Kbook}), the composition in \eqref{eqn:coker-alt-V*-2} is also multiplication by the degree $[k(z_0):k(x)]$. In particular, the cokernel of the 
    map $\oplus_{z \in (\phi')^{-1}(x)} k(z)^{\times} \to k(x)^{\times}$ is a torsion group of exponent $[k(z_0):k(x)]$. 
    Since $\phi' \colon C' \to C^N$ is a finite surjective morphism between regular integral curves, it follows from \cite[Example A.3.3]{Fulton} 
    that $[k(z_0):k(x)] \leq n$. (Indeed, by the formula $n = \sum_i e_i f_i$, we have $f_i \leq n$.) This proves the claim that the cokernel of the 
    map $\oplus_{z \in (\phi')^{-1}(x)} k(z)^{\times} \to k(x)^{\times}$ is a torsion group of exponent $n !$ and hence completes the proof of the lemma. 
\end{proof}

We now prove the main result of this section. 

\begin{thm} \label{thm:B-curve}
    Let $C$ be a geometrically irreducible 
     projective integral curve over a local field $k$. Then 
    \begin{equation} \label{eqn:thm-B-curve*-0}
        V(C/k) = D \oplus T,
    \end{equation}
    where  $D$ is a divisible group and $T$ is a torsion group of finite exponent. 
\end{thm}
\begin{proof}
    With notations as in \eqref{eqn:dia-curves*-1}, we have the following exact sequence. 
    \begin{equation} \label{eqn:thm-B-curve*-1}
        {\rm coker}(\psi') \to  {\rm coker}(\psi\circ \psi') \to  {\rm coker}(\psi). 
    \end{equation}
    In Lemmas \ref{lem:coker-nor-V} and \ref{lem:coker-alt-V}, we proved that ${\rm coker}(\psi)$ and ${\rm coker}(\psi')$ are torsion groups of finite exponents. The above exact sequences therefore yield that the cokernel of the map $\psi \circ \psi' \colon V(C'/k') \to V(C/k)$ is also a torsion group of finite exponent, say $n_1$. By \cite[Theorem~1.3]{Yos03}, we have $V(C'/k') = D' \oplus T'$, where $D'$ is a divisible group and $T'$ is a torsion group of exponent $n_2$. Let $D\subset V(C/k)$ be the image of $D'$ under the above map $\psi \circ \psi'$. Then it follows that $D$ is a divisible group and $V(C/k)/D$ is a torsion group of exponent $n_1 n_2$. But then $V(C/k) = D \oplus T$, where $T$ a torsion group of exponent $n_1 n_2$. This completes the proof. 
\end{proof}

Recall that we restated Theorem \ref{thm:B'} as Theorem \ref{thm:B}. We now provide a proof of the theorem. 

\subsection{Proof of Theorem \ref{thm:B'}}
Combine Lemma \ref{lem:red-to-curve} and Theorem \ref{thm:B-curve}. \qed

\section{Application: Unramified Class field theory for regular curves} \label{sec:CFT}

In this section, we give an application of our Theorems \ref{thm:A} and \ref{thm:B-curve}. We let $k$ denote a local field of positive characteristic $p>0$ and let $C$ be a regular projective curve over $k$. We then have a reciprocity map 
\begin{equation} \label{eqn:Rec-map}
    \rho_C \colon SK_1(C) \to \pi_1^{\ab}(C). 
\end{equation}
For the existence of the above reciprocity map see \cite{Saito85}, 
\cite[Section 1.5]{KS83} or \cite[Proof of Lemma~7.2]{TCFT-arXiv}. 
We start by understanding the kernel of the induced modulo $n$-map $\rho_{C,n} \colon SK_1(C)/n \to \pi_1^{\ab}(C)/n$ for $n\geq 1$. Recall that the map $\rho_{C,n}$ is injective if $C$ is smooth projective curve over $k$ by \cite[II, Lemma 5.3]{Saito85} (prime to $p$-case) and \cite[Proposition 3]{KS-UCFT} ($p$-primary case). 

Note that, in these references, the result is proven under the additional assumption that $C$ is a geometrically connected smooth projective curve over $k$. 
But if $C$ is a smooth projective curve over $k$, it is a geometrically connected smooth projective curve over $k'=H^0(X, \mathcal{O}_X)$ which is a 
 finite separable field extension over $k$ (because $C$ is geometrically reduced over $k$). Finally, we note that the claim that $\rho_{C,n}$ is injective is absolute and does not depend on the base field, and hence it also holds for smooth projective curves by the above references (see the proof of \cite[Lemma 11.15]{TCFT-arXiv}). 

We prove that the map $\rho_{C,n}$ is still injective when $C$ is a regular projective curve over a local field $k$ of characteristic $p>0$ and $(n,p)=1$. When $n=p^r$, we need certain additional assumptions on $C$. The main ingredients in both cases are duality theorems. 
We first start with prime to $p$-duality.

\subsection{Duality in prime to $p$-case} \label{sec:duality-I}

Let $k$ be a local field of characteristic $p>0$ and $(n,p)=1$. If $X$ is a smooth geometrically connected projective curve over $k$, then by Poincare duality for $X \otimes_k \overline{k}$ and Tate duality for Galois cohomology of $k$, there exist an isomorphism ${\rm Tr}_X\colon H^4(X, \Z/n(2)) \xrightarrow{\cong} \Z/n$
and for each $ i\geq 0$ and $j \in \Z$, a perfect pairing 
\begin{equation} \label{eqn:PP-Smooth}
    H^i(X, \Z/n(j)) \times H^{4-i}(X, \Z/n(2-j)) \to H^4(X, \Z/n(2)) \xrightarrow[\cong]{{\rm Tr}_X} \Z/n. 
\end{equation}
In this section, we prove that the same holds for a regular projective curve over $k$. 
We start with constructing the trace map. Recall that if $l$ is local field of characteristic $p$ and $(n,p)=1$, there exists an isomorphism (for example, see 
\cite[Theorem~III]{Kato80})
\begin{equation} \label{eqn:Tr-iso-LF}
    \Tr_l \colon H^2(l, \Z/n(1))  \xrightarrow{\cong} \Z/n.
\end{equation}

\begin{lem} \label{lem:Tr-reg-cur}
    Let $C$ be a regular projective integral curve over a local field $k$. Then there exists a surjective homomorphism $\Tr_C \colon  H^4(C, \Z/n(2)) \surj \Z/n$ such that for $x \in C_{(0)}$, the following diagram commutes. 
    \begin{equation} \label{eqn:Tr-reg-cur*-0}
        \xymatrix@C.8pc{
        H^4_x(C, \Z/n(2)) \ar[r] & H^4(C, \Z/n(2)) \ar[d]^-{\Tr_C} 
        \\ 
        H^2(k(x), \Z/n(1)) \ar[r]^-{\Tr_{k(x)}} \ar[u]^-{\cong} & \Z/n,
        }
    \end{equation}
    where the left vertical arrow is Gabber's purity isomorphism (see \cite{Fuji02}) for the closed embedding $\Spec(k(x)) \to C$ of regular schemes. 
\end{lem}
\begin{proof}
    Let $K = k(\eta)$, where $\eta$ is the generic point of $C$. Since \'etale cohomological dimension of $K$ is $\leq 3$ by \cite[Lemma~0F0T]{StackP}, we have $H^4(K, \Z/n(2)) =0$. 
    The localization sequence then yields the following exact sequence. 
    \begin{equation} \label{Tr-reg-cur*-1}
        H^3(K, \Z/n(2)) \to \oplus_{x\in C_{(0)}} H^4_x(C, \Z/n(2)) \to  H^4(C, \Z/n(2)) \to 0. 
    \end{equation}
    Since $C$ is regular curve, for $x\in C_{(0)}$, Gabber's purity theorem yields the isomorphism 
    $H^4_x(C, \Z/n(2)) \xleftarrow{\cong} H^2(k(x), \Z/n(1))$. In particular, the isomorphisms $\Tr_{k(x)}$ of \eqref{eqn:Tr-iso-LF} yield isomorphisms
    $\Tr_x \colon H^4_x(C, \Z/n(2)) \xrightarrow{\cong} \Z/n$ for $x\in C_{(0)}$
    such that the right upper square in the following diagram commutes. 
    \begin{equation}\label{Tr-reg-cur*-2}
    \xymatrix@C2pc{
        H^3(K, \Z/n(2)) \ar[r] \ar@{=}[dd]& \oplus_{x\in C_{(0)}} H^4_x(C, \Z/n(2)) 
        \ar[r]^-{\sum_x \Tr_x} & \Z/n\ar@{=}[d] \\
        & \oplus_{x\in C_{(0)}} H^2(k(x), \Z/n(1)) \ar[u]^-{\cong}
         \ar[r]^-{\sum_x \Tr_{k(x)}} & \Z/n \ar@{=}[d] \\
        H^3(K, \Z/n(2)) \ar[r] & \oplus_{x\in C_{(0)}}H^3(K_x, \Z/n(2)) 
        \ar[u]^-{\partial}_-{\cong} \ar[r]^-{\sum_x \Tr_{K_x}}& \Z/n,
        }
    \end{equation}
    where $K_x$ is the function field of the completion of $\mathcal{O}_{X, x}$
    %the bottom right horizontal arrow is obtained by the Kato isomorphisms $\Tr_{K_x}$ for the $2$-local fields $K_x$ from \cite[Theorem~III]{Kato80}, 
    and the bottom middle vertical arrow $\partial$ is induced by the Kato boundary maps \cite[Section~1.2, Page~610]{Kato80}. It follows from the proof of Theorem~3 (on Page 625) of loc. cit. that  $\partial$ is an isomorphism 
    %by 
    %a cohomological dimensional argument and by 
    %Section~3.2, Lemma~2 of loc. cit. 
    and there exist isomorphisms $\Tr_{K_x}$ such that the bottom right square commutes. 
    It follows from  \cite[Section~5, Page~120]{KS83} that the natural map  $H^3(K, \Z/n(2)) \to \prod_{x\in C_{(0)}}H^3(K_x, \Z/n(2))$ factors through  $\oplus_{x\in C_{(0)}}H^3(K_x, \Z/n(2))$ and 
    the composition of the arrows in the resulting bottom row in \eqref{Tr-reg-cur*-2} is zero.  
    
    We now assume that the left square in \eqref{Tr-reg-cur*-2} commutes and complete the proof. %by Lemma \ref{eqn:left-square-6.6}. 
    %It follows from     \cite[Lemma~1.4(2)]{Kato-Hasse} that the composition of the middle vertical arrows is the same as the composition of the maps. 
    Since the outer square in \eqref{Tr-reg-cur*-2} then commutes 
    and  the composition of the arrows in the bottom row is zero,  
    it follows that the composition of the arrows in the top row is also  zero.  %Since the outer square in \eqref{Tr-reg-cur*-2} commutes and the bottom-row composition is zero, the top-row composition is also zero
    By the exact sequence \eqref{Tr-reg-cur*-1}, $H^4(C, \Z/n(2))$ is the cokernel of the top left horizontal arrow in  \eqref{Tr-reg-cur*-2}. We therefore get a homomorphism 
    ${\rm Tr}_C \colon H^4(C, \Z/n(2)) \to \Z/n$ such that the composition of the maps $ H^4_x(C, \Z/n(2)) \to H^4(C, \Z/n(2)) \xrightarrow{{\rm Tr}_C} \Z/n$ is the same as the map ${\rm Tr}_{x}$. The commutativity of the top right square in \eqref{Tr-reg-cur*-2} yields that the diagram \eqref{eqn:Tr-reg-cur*-0} commutes. 
    %, it suffices to show that the composition of the arrows in the top row in \eqref{Tr-reg-cur*-2} is zero. Equivalently, it suffices to show that the composition of the arrows in the bottom row in \eqref{Tr-reg-cur*-2} is zero.     But this follows from \cite[Section~5, Page~120]{KS83}.
    We are now left to show that the left square in \eqref{Tr-reg-cur*-2} commutes.  
  
    To see this, we note that the top left horizontal arrow in \eqref{Tr-reg-cur*-2}, coming from the localization sequences for $C$, is the same as the composition of the maps $  H^3(K, \Z/n(2)) \to \oplus_{x \in C_{(0)}}  H^3(K^h_x, \Z/n(2)) \xrightarrow{\alpha} \oplus_{x \in C_{(0)}} H^4_x(C, \Z/n(2))$, where $K_x^h$ denotes the henselization of $K$ at the closed point $x \in C$ and the last arrow $\alpha$ is induced by the localization sequences for henselian local rings $\mathcal{O}_{C, x}^h$ and excision. 
   For $x \in C_{(0)}$, \cite[Page~149]{Kato-Hasse} yields a boundary map  
    $\partial' \colon H^3(K^h_x, \Z/n(2)) \to  H^2(k(x), \Z/n(1))$ such that we have the  following diagram. 
    \begin{equation}\label{eqn:left-squ-6.6-I}
    \xymatrix@C2pc{
         H^3(K, \Z/n(2)) \ar@{=}[d] \ar[r] & H^3(K^h_x, \Z/n(2)) \ar[r]^-{\partial'} \ar[d]    &  H^2(k(x), \Z/n(1)) )\ar[r]^-{\cong}\ar@{=}[d]&  H^4_x(C, \Z/n(2))\\
        H^3(K, \Z/n(2)) \ar[r] & H^3(K_x, \Z/n(2)) \ar[r]^-{\partial} &  H^2(k(x), \Z/n(1)).}
    \end{equation}
    The middle square in \eqref{eqn:left-squ-6.6-I} commutes by the definition of middle horizontal arrows (compare \cite[Page~149]{Kato-Hasse} with  \cite[Section~1.2, Page~610]{Kato80}). Moreover, 
\cite[Lemma~1.4]{Kato-Hasse} asserts that, taking sum over all $x\in C_{(0)}$, the composition of the second and the third arrows in the top row in \eqref{eqn:left-squ-6.6-I} is the same as the map the above map $\alpha$ and in particular, the composition of all arrows in  the top row  in \eqref{eqn:left-squ-6.6-I} (taking sum over all $x\in C_{(0)}$) is the same as the top left horizontal arrow in  \eqref{Tr-reg-cur*-2}. 
   The commutativity of the  left square in \eqref{Tr-reg-cur*-2} therefore follows from the commutative diagram  \eqref{eqn:left-squ-6.6-I}. This completes the proof. 
    \end{proof}

We now prove that the map $\Tr_C$ in Lemma~\ref{lem:Tr-reg-cur} is an isomorphism. Instead of directly proving that the surjective map $\Tr_C$ is also injective, we prove that $H^4(C, \Z/n(2))$ injects into $\Z/n$. But this is sufficient because if there is a surjection $\lambda\colon H \surj \Z/n$ such that $H$ is a subgroup of $\Z/n$, then $H = \Z/n$ and the map $\lambda$ is an isomorphism.

\begin{lem}\label{lem:Tr-reg-iso}
    Let notations be as in Lemma~\ref{lem:Tr-reg-cur}. Then the map $\Tr_C \colon  H^4(C, \Z/n(2)) \surj \Z/n$ is an isomorphism. 
\end{lem}
\begin{proof}
    As mentioned above, it suffices to prove the existence of an injective homomorphism  $H^4(C, \Z/n(2)) \inj \Z/n$. To see this, we let $\pi\colon W \to C$ be an alteration of $C$ of degree $p^r$ for $r \geq 0$. 
    Recall that a dominant and generically finite proper morphism between regular integral curves is a finite flat morphism (see discussion before Lemma~\ref{lem:coker-alt-V}). 
    We therefore conclude that $\pi$ is a finite flat morphism of degree $p^r$  between regular projective curves and $W$ is smooth projective geometrically connected over a finite field extension $l$ over $k$. By \eqref{eqn:PP-Smooth}, it follows that ${\Tr}_W \colon H^4(W, \Z/n(2)) \xrightarrow{\cong} \Z/n$. It therefore suffices to show that the pull-back map $ \pi^* \colon H^4(C, \Z/n(2))  \to H^4(W, \Z/n(2)) $ is injective. To see this, we observe that $\pi$ is finite and we therefore have a push-forward map $\pi_* \colon H^4(W, \Z/n(2)) \to H^4(C, \Z/n(2)) $ such that $\pi_* \circ \pi^* = p^r \colon  H^4(C, \Z/n(2)) \to  H^4(C, \Z/n(2))$. Since $n$ is co-prime to $p$, the multiplication $p^r$-map is an isomorphism. This implies that the map $\pi^*$ is injective and completes the proof. 
\end{proof}

The cup-product of \'etale cohomology groups and the isomorphism in Lemma \ref{lem:Tr-reg-iso} yield the following pairing for $i \geq 0$ and $j \in \Z$.  
\begin{equation} \label{eqn:Pairing-reg}
    H^i(C, \Z/n(j)) \times H^{4-i}(C, \Z/n(2-j)) \to H^4(C, \Z/n(2)) \xrightarrow[\cong]{{\rm Tr}_C} \Z/n. 
\end{equation}
We next prove that the above pairing is perfect. Since all the above cohomology groups are finite groups of exponent $n$ (see \cite[Remark 3.5 (c)]{Jannsen}), it suffices to show that for all $i \geq 0$ and $j \in \Z$, the induced map $H^i(C, \Z/n(j)) \to
{\rm Hom}(H^{4-i}(C, \Z/n(2-j)), H^4(C, \Z/n(2)))$ is injective. This is the main step in proving the following theorem. 

\begin{thm} \label{thm:PP-reg}
    Let $C$ be a regular projective integral curve over a local field $k$. Then the pairing \eqref{eqn:Pairing-reg} is a perfect pairing of finite abelian groups. 
\end{thm}
\begin{proof}
    We let $\pi \colon W \to C $ be an alteration of $C$ of degree $p^r$ for $r \geq 0$. Since $\pi$ is a finite flat morphism, the projection formula for $\pi_*$ yields the following commutative diagram. 
    \begin{equation} \label{eqn:PP-reg*-1}
       \xymatrix@C1.5pc{
        H^i(W, \Z/n(j)) \times H^{4-i}(W, \Z/n(2-j)) \ar[r]
        \ar@<10ex>@{->>}[d]^-{\pi_*}&  H^4(W, \Z/n(2))  \ar[d]^-{\pi_*}_-{\cong}
       \ar[r]_-{\cong}^-{{\rm Tr}_W} & \Z/n\\
        H^i(C, \Z/n(j)) \times H^{4-i}(C, \Z/n(2-j)) \ar[r] 
        \ar@<10ex>@{^{(}->}[u]^-{\pi^*}&  H^4(C, \Z/n(2)) \ar[r]_-{\cong}^-{{\rm Tr}_C} & \Z/n, 
       }
    \end{equation}
    where $i \geq 0$ and $j \in \Z$. 
    %We saw in the proof of Lemma \ref{lem:Tr-reg-iso} that the composition of pull-back and pushforward maps, i.e.  
    Since $\pi_* \circ \pi^* = p^r$ and $n$ is co-prime to $p$, the left vertical arrow in \eqref{eqn:PP-reg*-1} is injective, and the middle and right vertical arrows are surjective. %We have seen in the proof of Lemma \ref{lem:Tr-reg-iso} that the right vertical arrow in \eqref{eqn:PP-reg*-1} is a surjective map and 
    By Lemma \ref{lem:Tr-reg-iso}, the domain and the codomain of the right vertical map are isomorphic to $\Z/n$. It therefore follows that the right vertical arrow in \eqref{eqn:PP-reg*-1} is an isomorphism.  As before, there exists a finite extension $k'$ over $k$ such that $W$ is a smooth projective geometrically connected curve over $k'$.
    Then by \eqref{eqn:PP-Smooth}, there exists an isomorphism 
    $H^4(W, \Z/n(2))  \xrightarrow{\cong} \Z/n$ such that the induced pairing 
   $H^i(W, \Z/n(j)) \times H^{4-i}(W, \Z/n(2-j)) \to  H^4(W, \Z/n(2))  
       \xrightarrow{\cong} \Z/n$ 
    is  a perfect pairing of finite groups. In particular, the map $H^i(W, \Z/n(j)) \xrightarrow{\cong} {\rm Hom}(H^{4-i}(W, \Z/n(2-j)), H^4(W, \Z/n(2)))$ is an isomorphism for each $i\geq 0$ and $j \in \Z$. 
    Since the left vertical arrow in \eqref{eqn:PP-reg*-1} is injective, the induced map 
    $H^i(C, \Z/n(j)) \inj {\rm Hom}(H^{4-i}(C, \Z/n(2-j)), H^4(C, \Z/n(2)))$ is injective for each $i\geq 0$ and $j \in \Z$. By a cardinality argument, we obtain that the above map is an isomorphism for each $i\geq 0$ and $j \in \Z$. This completes the proof. 
\end{proof}

\begin{cor} \label{cor:H-3==pi-1}
    Let $C$ be a regular projective integral curve over a local field $k$. Then there exists an isomorphism $\theta_C \colon H^3(C, \Z/n(2)) \xrightarrow{\cong} {\rm Hom}(H^1(C, \Z/n), \Z/n) \cong \pi^{\ab}_1(C)/n$ such that for $x\in C_{(0)}$, its composition with the arrows
    \begin{equation} \label{eqn:H-3==pi-1*1}
        k(x)^{\times}/n \xrightarrow{\psi_{x,n}}   H^1(k(x), \Z/n(1)) \xrightarrow[\cong]{(\iota_x)_*}  H^3_x(C, \Z/n(2)) \to  H^3(C, \Z/n(2))
    \end{equation}
    is the same as the composition $k(x)^{\times}/n \xrightarrow{\rho_x} \pi_1^{\ab}(k(x))/n \to \pi_1^{\ab}(C)/n$, where $\rho_x$ is the reciprocity map for the local field $k(x)$ and the map $\psi_{x,n}$ is the Norm residue map. 
\end{cor}
\begin{proof}
    For $x\in C_{(0)}$, we consider the following diagram of pairings.  
    \begin{equation}\label{eqn:H-3==pi-1*2}
        \xymatrix@C2pc{
        H^3(C, \Z/n(2)) \times H^{1}(C, \Z/n) \ar[r] 
         \ar@<10ex>[d]^-{\iota_x^*}
        &  H^4(C, \Z/n(2)) \ar[r]_-{\cong}^-{{\rm Tr}_C} & \Z/n \ar@{=}[d]\\
      %  H^3_x(C, \Z/n(2)) \hspace{5pc} & H^4_x(C, \Z/n(2)) \\
        H^1(k(x), \Z/n(1)) \times H^{1}(k(x), \Z/n) \ar[r] 
        \ar@<7ex>[u]^-{(\iota_x)_*}
        &  H^2(k(x), \Z/n(1)) \ar[r]_-{\cong}^-{{\rm Tr}_{k(x)}} 
         \ar[u]^-{(\iota_x)_*} & \Z/n,  }
    \end{equation}
    where $\iota_x \colon \Spec(k(x)) \to C$ is the inclusion of the closed point. Note that the left and the right vertical arrows are respective Gysin maps. Recall from Lemma~\ref{lem:Tr-reg-cur} that the right square in \eqref{eqn:H-3==pi-1*2} commutes by the definition of $\Tr_C$. Moreover, the left diagram commutes because the Gysin map is $H^*(C, \Z/n)$-linear with its action on $H^*(k(x), \Z/n)$ via the pull-back map by the proof of \cite[VI, Proposition~6.5]{Milne}. Indeed, in loc. cit., the same result is proven for closed immersion of smooth schemes (over a separable field) but they only use that the Gysin map on cohomology groups is induced by an isomorphism $\mathcal{H}^2_x(\Z/n) \cong \Z/n(-1)$ which is known by \cite{Fuji02} (see the diagram on \cite[VI, Pages~250-251]{Milne}). 
    Since ${\rm Hom}(H^1(C, \Z/n), \Z/n) \cong \pi_1^{\ab}(C)/n$, the commutative diagram \eqref{eqn:H-3==pi-1*2} of perfect pairings gives the following commutative square. 
    \begin{equation}\label{eqn:H-3==pi-1*3}
        \xymatrix@C2pc{
       & H^3(C, \Z/n(2)) \ar[r]^-{\theta_C}_-{\cong}& \pi^{\ab}_1(C)/n\\
      k(x)^{\times}/n \ar[r]^-{\psi_{x,n}} &  H^1(k(x), \Z/n(1)) \ar[r]^-{\theta_x} \ar[u]^-{(\iota_x)_*} & \pi^{\ab}_1(k(x))/n \ar[u]^-{(\iota_x)_*}. 
        }
    \end{equation}
  The lemma now follows because the composition of the arrows in the bottom row in \eqref{eqn:H-3==pi-1*3} is the same as the reciprocity map for the local field $k(x)$. 
\end{proof}

\begin{thm} \label{thm:Res-inj-mod-n}
    Let $C$ be a regular projective integral curve over a local field $k$ of positive characteristic $p>0$ and let $n \geq 0$ be an integer such that $(p,n)=1$. Then the induced map $\rho_{C,n} \colon SK_1(C)/n \to \pi_1^{\ab}(C)/n$ is injective. 
\end{thm}
\begin{proof}
   As before, we let $K$ denote the field of functions of $C$. We consider the following diagram. 
    \begin{equation} \label{eqn:Res-inj-mod-n*-1}
    \xymatrix@C.8pc{
        K_2(K)/n \ar[r] \ar[d]^-{\psi_{K,n}}_-{\cong}& \oplus_{x\in C_{(0)}} k(x)^{\times}/n \ar[r] \ar[d]^-{\psi_{x,n}}_-{\cong} &  SK_1(C)/n \ar[r] 
        \ar@{-->}[dd]& 0\\
        H^2(K, \Z/n(2)) \ar[r] \ar@{=}[d]& \oplus_{x\in C_{(0)}} H^1(k(x), \Z/n(1)) \ar[d]^-{(\iota_x)_*}_-{\cong}& \\
         H^2(K, \Z/n(2)) \ar[r] & \oplus_{x\in C_{(0)}} H^3_x(C, \Z/n(2)) \ar[r] & H^3(C, \Z/n(2))  \ar[d]^-{\theta_C}_-{\cong} & \\
         & &  \pi_1^{\ab}(C)/n,
        }
    \end{equation}
    where the top left and middle vertical arrows are Norm residue maps. We first note that the top row is exact by \eqref{eqn:def-SK_1*-1} and by the right exactness of the tensor product, whereas the bottom row is exact by the localisation sequence for étale cohomology.  It is known that the left outer square commutes. The Kummer sequence implies that the map $\psi_{x,n}$ is an isomorphism, and the map $(\iota_{x})_{\ast}$ is an isomorphism by the Gabber's purity theorem \cite{Fuji02}. This shows that the middle vertical arrows are isomorphisms, and the left top vertical arrow is an isomorphism by the Merkurjev-Suslin Theorem \cite{Merkurjev-Suslin}. An easy diagram chase tells us that the dotted arrow exists and is injective. But then we are done because by 
    Corollary~\ref{cor:H-3==pi-1} the composition of the dotted arrow with $\theta_C$ is the same as the reciprocity map $\rho_{C,n} \colon SK_1(C)/n \to \pi_1^{\ab}(C)/n$. This completes the proof of the theorem. 
\end{proof}

\subsection{Duality in $p$-primary case} \label{sec:duality-II}

Let $C$ and other notations be as in Section \ref{sec:duality-I}. Assume further that $C$ is geometrically connected and is generically smooth (for example, $C$ is the normalization of a compactification of a geometrically connected smooth quasi-projective curve over $k$). Then for $i, j \geq 0$ and $m \geq 1$, \cite[Theorem~4.7]{KRS23} gives us the following pairing of topological abelian groups. 
\begin{equation} \label{eqn:PP-II*-1.}
    H^i(C, W_m \Omega^{j}_{C, \log}) \times H^{2-i}(C, W_m \Omega^{2-j}_{C, \log}) \to H^{2}(C, W_m \Omega^{2}_{C, \log}) \xrightarrow[\cong]{{\rm Tr}_C} \Z/p^m. 
\end{equation}
When $i = 1$ and $j =2$, the above pairing gives the following perfect pairing between a profinite abelian group and a discrete torsion group (see \cite[Proposition~3.6, Theorem~4.7]{KRS23}).
\begin{equation} \label{eqn:PP-II*-1}
    H^1(C, W_m \Omega^{2}_{C, \log}) \times H^{1}(C, \mathbb{Z}/p^{m}) \to H^{2}(C, W_m \Omega^{2}_{C, \log}) \xrightarrow[\cong]{{\rm Tr}_C} \Z/p^m. 
\end{equation}

As a corollary, we get
\begin{cor} \label{cor:PP-II}
    There exists an isomorphism 
    \[
    \theta'_C \colon H^1(C, W_m \Omega^2_{C, \log}) \xrightarrow{\cong} {\rm Hom}(H^1(C, \Z/p^m), \Z/p^m) \cong \pi^{\ab}_1(C)/p^m
    \]
    such that for $x\in C_{(0)}$, its composition with the arrows
    \begin{equation} \label{eqn:PP-II*-2}
        k(x)^{\times}/p^m \xrightarrow{{\rm d log}}   H^0(k(x), W_m \Omega^1_{k(x), \log}) \xrightarrow[\cong]{(\iota_x)_*}   H^1_x(C, W_m \Omega^2_{C, \log}) \to   H^1(C, W_m \Omega^2_{C, \log})    \end{equation}
    is the same as the composition $k(x)^{\times}/p^m \xrightarrow{\rho_x} \pi_1^{\ab}(k(x))/p^m \to \pi_1^{\ab}(C)/p^m$, where $\rho_x$ is the reciprocity map for the local field $k(x)$ and the middle arrow  is the isomorphism given by \cite{Shiho}. 
\end{cor}
\begin{proof}
      For $x\in C_{(0)}$, we consider the following diagram of pairings.  
    \begin{equation}\label{eqn:PP-II*-3}
        \xymatrix@C2pc{
        H^1(C, W_m \Omega^2_{C, \log}) \times H^{1}(C, \Z/p^m) \ar[r] 
         \ar@<10ex>[d]^-{\iota_x^*}
        &  H^2(C, W_m \Omega^2_{C, \log}) \ar[r]_-{\cong}^-{{\rm Tr}_C} & \Z/p^m \ar@{=}[d]\\
      %  H^3_x(C, \Z/n(2)) \hspace{5pc} & H^4_x(C, \Z/n(2)) \\
        H^0(k(x), W_m \Omega^1_{k(x), \log}) \times H^{1}(k(x), \Z/p^m) \ar[r] 
        \ar@<7ex>[u]^-{(\iota_x)_*}
        & H^1 (k(x), W_m \Omega^1_{k(x), \log}) \ar[r]_-{\cong}^-{{\rm Tr}_{k(x)}} 
         \ar[u]^-{(\iota_x)_*} & \Z/p^m.}
    \end{equation}
    By \cite[Proposition~4.4]{KRS23}, the right square in \eqref{eqn:PP-II*-3} is commutative. Note that the leftmost vertical arrow and the third from the left vertical arrow are defined using Gysin isomorphisms 
    $H^i(k(x), W_m \Omega^j_{k(x), \log}) \xrightarrow{\cong} H^{i+1}_x(C, W_m\Omega^{j+1}_{C,\log})$ for $i=0, 1$ and $j=1$ of \cite{Shiho}. By the loc. cit. (see Pages~587, 590), the above isomorphisms are induced by the isomorphism 
    $\rho^{j+1,\log}_{\iota, m}\colon W_m \Omega^j_{k(x), \log} \xrightarrow{\cong} \mathcal{H}_x^1(W_m \Omega^{j+1}_{C, \log})$. We can now follow the steps of the proof of \cite[VI, Proposition~6.5]{Milne} to conclude that the left diagram in \eqref{eqn:PP-II*-3} commutes. We therefore get the following commutative square. 
    %such that the horizontal arrows are isomorphisms. 
  %  of perfect pairing gives the following commutative square. 
    \begin{equation}\label{eqn:PP-II*-4}
        \xymatrix@C2pc{
       &H^1(C, W_m \Omega^2_{C, \log}) \ar[r]^-{\theta'_C}_-{\cong}& \pi^{\ab}_1(C)/p^m\\
      k(x)^{\times}/p^m \ar[r]^-{{\rm d log}} & H^0(k(x), W_m \Omega^1_{k(x), \log}) \ar[r]^-{\theta'_x} \ar[u]^-{(\iota_x)_*} & \pi^{\ab}_1(k(x))/p^m \ar[u]^-{(\iota_x)_*}. 
        }
    \end{equation}
  The lemma now follows because the composition of the arrows in the bottom row in \eqref{eqn:PP-II*-4} is the same as the reciprocity map for the local field $k(x)$. 
\end{proof}

\begin{rem}
    The generic smoothness of the curve $C$ in the above corollary is used in the proof of the isomorphism $\theta'_{C}$, which was established in \cite[Theorem~4.7]{KRS23} (see also Remark $1.5$ in loc. cit.).
\end{rem}

\begin{thm} \label{thm:Res-inj-mod-p}
    Let $C$ be a regular geometrically connected projective integral curve over a local field $k$ of positive characteristic $p>0$ and let $m \geq 0$ be an integer. Assume that $C$ is generically smooth. Then the induced map $\rho_{C,p^m} \colon SK_1(C)/p^m \to \pi_1^{\ab}(C)/p^m$ is injective. 
\end{thm}
\begin{proof}
   As before, we let $K$ denote the field of functions of $C$. We consider the following diagram of exact sequences. 
    \begin{equation} \label{eqn:Res-inj-mod-p*-1}
    \xymatrix@C.8pc{
        K_2(K)/p^m \ar[r] \ar[d]^-{{\rm d log}}_-{\cong}& \oplus_{x\in C_{(0)}} k(x)^{\times}/p^m \ar[r] \ar[d]^-{{\rm dlog}}_-{\cong} &  SK_1(C)/p^m \ar[r] 
        \ar@{-->}[dd]& 0\\
        H^0(K, W_m \Omega^2_{K, \log})  \ar@{=}[d]& \oplus_{x\in C_{(0)}}H^0(k(x), W_m \Omega^1_{k(x), \log}) \ar[d]^-{(\iota_x)_*}_-{\cong}& \\
     H^0(K, W_m \Omega^2_{K, \log}) \ar[r] & \oplus_{x\in C_{(0)}} 
     H^1_x(C, W_m \Omega^2_{C, \log}) \ar[r] & H^1(C, W_m \Omega^2_{C, \log}) \ar[d]^-{\theta'_C}_-{\cong}\\
         & &  \pi_1^{\ab}(C)/p^m,
        }
    \end{equation}
    where the top left and the top middle vertical arrows are dlog-maps. Let us first assume that the left square in \eqref{eqn:Res-inj-mod-p*-1} commutes. By \cite[Corollary 2.8]{Bloch-Kato}, the left vertical and the top middle vertical arrows are isomorphisms. The bottom middle vertical arrow is an isomorphism by Corollary \ref{cor:PP-II} and hence a diagram chase tells us that the dotted arrow exists and is injective. 
 
    But then we are done because by Corollary~\ref{cor:PP-II}, the composition of the dotted arrow with $\theta'_C$ is the same as the reciprocity map $\rho_{C,p^m} \colon SK_1(C)/p^m \to \pi_1^{\ab}(C)/p^m$. We are therefore left to show that the left square in \eqref{eqn:Res-inj-mod-p*-1} commutes. To see this, let $x\in C_{(0)}$ be a closed point of $C$. We then have to show that the following diagram commutes. 
     \begin{equation} \label{eqn:Res-inj-mod-p*-2}
    \xymatrix@C.8pc{
        K_2(K_x)/p^m \ar[r] \ar[d]^-{{\rm d log}}&  k(x)^{\times}/p^m  \ar[d]^-{{\rm dlog}} \\
        H^0(K_x, W_m \Omega^2_{K_x, \log})  \ar@{->>}[d]& H^0(k(x), W_m \Omega^1_{k(x), \log}) \ar[d]^-{(\iota_x)_*}& \\
     \frac{H^0(K_x, W_m \Omega^2_{K_x, \log})}{H^0(R_x, W_m \Omega^2_{R_x, \log})} \ar[r]^-{\cong} &      H^1_x(C, W_m \Omega^2_{C, \log}),
        }
    \end{equation}
    where $R_x$ is the henselization of $\mathcal{O}_{C,x}$ and $K_x$ is the function field of $R_x$. But this follows from the definition of the Gysin map (the right bottom vertical arrow) by \cite[Page~587]{Shiho}. Indeed, by the loc. cit., the composition of the right vertical arrows with the inverse of the bottom horizontal arrow (that is an isomorphism) is the map $k(x)^{\times}/p^m  \to  \frac{H^0(K_x, W_m \Omega^2_{K_x, \log})}{H^0(R_x, W_m \Omega^2_{R_x, \log})}$ so that  ${u} \mapsto {\rm dlog}[t]_m \wedge {\rm dlog}[\widetilde{u}]_m$, where  $\widetilde{u}\in R_x \subset K_x$ is a lift of $u\in k(x)^{\times}$ and $t$ is a unifomalizer of the regular local ring $R_x$. Moreover, the Milnor $K$-group $K_2(K_x)$ is generated by elements of the form $\{a_1, a_2\}$ and $\{t, a\}$, where $a_i, a \in R_x^{\times}$. It is well known that the composition of the maps $K_2(R_x) \to K_2(K_x) \to k(x)^{\times}$ is zero (for example, see \cite[III, Lemma~6.3]{Wei-Kbook}). Moreover, if $a_i\in R_x^{\times}$, then  ${\rm dlog}(\{a_1, a_2\}) = {\rm dlog}[a_1]_m \wedge {\rm dlog}[a_2]_m \in  H^0(K_x, W_m \Omega^2_{K_x, \log})$ lies in the image of $ H^0(R_x, W_m \Omega^2_{R_x, \log})$. 
    %${\rm dlog}(K_2(R_x)/p^m) \subset H^0(R_x, W_m \Omega^2_{R_x, \log})$.
    We therefore conclude that 
    the elements of the form $\{a_1, a_2\}$ (with $a_i\in R_x^{\times}$) go to zero by both the top horizontal arrow 
    and the composition of the left vertical arrows in \eqref{eqn:Res-inj-mod-p*-2}. By the definition of the left top vertical arrow (the ${\rm dlog}$-map), it therefore suffices to show that 
    the top horizontal arrow maps $\{t, a\}$ to $\overline{a} \in k(x)^*/p^m$ for all $a\in R_x^{\times}$. But this is a well-known property of the boundary map of Milnor $K$-groups (for example, see \cite[III, Lemma~6.3]{Wei-Kbook}). This completes the proof. 
\end{proof}

\subsection{Proof of Theorem~\ref{thm:CF-Reg}} \label{sec:Proof-CFT}
We let the notations be as in Theorem~\ref{thm:CF-Reg}. 
\begin{enumerate}
    \item 
 We look at the commutative diagram of exact sequences
 	\begin{equation} \label{eqn:Struc-geom-*1.7}
	\xymatrix@C.8pc{
0 \ar[r]  &	\Ker(\rho_X) \ar[r] \ar[d]^-{\alpha} & SK_{1}(X) \ar[r]^-{\rho_X} \ar[d]^-{\beta} & \pi^{\ab}_1(X) \ar[d]^-{\gamma}& \\
	0 \ar[r] &\varprojlim_{n \in \mathbb{L} } \Ker(\rho_{X,n}) \ar[r] &  \varprojlim_{n \in \mathbb{L}}  SK_{1}(X)/n  \ar[r]& \varprojlim_{n \in \mathbb{L} } \pi^{\ab}_1(X)/n,
	}
\end{equation}
where the bottom row is induced from the top row by  $\mathbb{L}$-completion. Here, we denote by $\mathbb{L}$ the set of all natural numbers $n$ with $(n, p) = 1$. By Theorem \ref{thm:Res-inj-mod-n}, we have $\varprojlim_{n \in \mathbb{L}} \Ker(\rho_{X,n}) = 0$  and the right bottom group has finite torsion by \eqref{eqn:prime-to-p-*2} and Theorem \ref{thm:A}. 
 We apply \cite[Lemma 7.7]{Jannsen-Saito} with $A = SK_{1}(X)$ and $B_{n} = \pi^{\ab}_1(X)/n$, to conclude that the group $\Ker(\beta)$ is $p'$-divisible.  We  consider the short exact sequence
\begin{equation}\label{Ker-Coker-ses}
    0 \rightarrow \Ker(\alpha) \xrightarrow{i} \Ker(\beta) \rightarrow \Coker(i) \rightarrow 0,
\end{equation}
   where $i$ is the map induced on the kernels of $\alpha$ and $\beta$. Let $n \in \mathbb{L}$ be any integer. Note that $\Coker(i) \subseteq \Ker(\gamma)$. Since $\Ker(\gamma) \cong \pi^{\ab}_1(X)_{p}$, and hence $\Ker(\gamma)  [n] = 0$, it follows that $\Coker(i)[n] = 0$. Tensoring the exact sequence \eqref{Ker-Coker-ses} with $\mathbb{Z}/n$ yields the following long exact sequence.
   \begin{equation*}\label{Ker-Coker-les}
   \Coker(i)[n] \rightarrow  \Ker(\alpha)/n \xrightarrow{i} \Ker(\beta)/n \rightarrow \Coker(i)/n \rightarrow 0.
\end{equation*} 
Since $\Coker(i)[n] = 0$ and $\Ker(\beta)/n = 0$, it follows that $\Ker(\alpha)/n = 0$. In particular, $\Ker(\alpha)$ is $p'$-divisible. By Theorem \ref{thm:Res-inj-mod-n}, we get $ \Ker(\rho_{X}) =  \Ker(\alpha)$, which implies that $ \Ker(\rho_{X})$ is $p'$-divisible. 
\vspace{1mm}

We now let $X$ be generically smooth and $\mathbb{L} = \mathbb{N}$. By  Theorems \ref{thm:Res-inj-mod-n} and \ref{thm:Res-inj-mod-p}, we have 
$\varprojlim_{n \in \mathbb{N} } \Ker(\rho_{X,n}) = 0$. Theorem \ref{thm:A} implies that the bottom right group  in \eqref{eqn:Struc-geom-*1.7} has finite torsion.  As before, applying \cite[Lemma 7.7]{Jannsen-Saito}, we then have $\Ker(\beta) = \bigcap_{n=1}^\infty n\ SK_{1}(X)$ is the maximal divisible subgroup of $SK_{1}(X)$. It suffices to show that $\Ker(\rho_X) =  \Ker(\beta)$. But this follows because 
the left vertical arrow $\alpha$ is zero and $\gamma$ is an isomorphism.

%Since $\mathbb{L} = \mathbb{N}$, the map $\gamma$ is an isomorphism. In particular, we have $\Ker(\rho_X) = \Ker(\alpha) = \Ker(\beta) = \bigcap_{n=1}^\infty n\ SK_{1}(X)$. is the maximal divisible subgroup of $SK_{1}(X)$.that $\Ker(\rho_X)$ is the maximal divisible subgroup of $SK_{1}(X)$. Indeed, the theorems imply that  Since $\mathbb{L} = \mathbb{N}$, the map $\gamma$ is an isomorphism. In particular, by \cite[Lemma 7.7]{Jannsen-Saito}, the subgroup $\Ker(\rho_X) = \Ker(\alpha) = \Ker(\beta) = \bigcap_{n=1}^\infty n\ SK_{1}(X)$ is the maximal divisible subgroup of $SK_{1}(X)$.

\item Consider the map  $\rho_{X,0} \colon V(X/k) \to \pi^{\ab}_1(X)^{{\rm geom}}$.  Since any map from a divisible group to a profinite abelian group is zero, the result follows from combining Theorems \ref{thm:A} and \ref{thm:B}.

\item We now consider the following commutative diagram. 
\begin{equation} \label{eqn:Struc-geom-*1.8}
	\xymatrix@C.8pc{
0 \ar[r]  &	V(X/k) \ar[r] \ar[d]^-{\rho_{X,0}} & SK_{1}(X) \ar[r]^-{N_X} \ar[d]^-{\rho_X} & k^{\times} \ar[d]^-{\rho_k}& \\
	0 \ar[r] & \pi^{\ab}_{1}(X)^{\geom} \ar[r] &  \pi^{\ab}_{1}(X)  \ar[r]& G^{\ab}_{k}.	}
\end{equation}
Since the reciprocity map $\rho_k$ is injective, we have the following short exact sequence. 
\begin{equation*}
0 \to \Coker(\rho_{X,0}) \to \Coker(\rho_{X}) \to \Coker(\rho_{k} \circ N_X) . 
\end{equation*}
 By \cite[XIV, Section 6, Corollary 2]{Serre} (also see the last diagram on Page~198 of loc. cit.), we have  $\Coker(\rho_k) \cong \widehat{\mathbb{Z}}/\mathbb{Z}$, and hence it has no torsion. Since the map $N_X$ has finite index and the group $\Coker(\rho_k)$ has no torsion, it follows that the group $\Coker(\rho_{k} \circ N_X)_{\tor}$ is finite. It therefore suffices to show that the group $\Coker(\rho_{X,0})_{\tor}$ is finite. Since the image of the map $\rho_{X,0}$ is finite by part (2), we have a surjection $\pi^{\ab}_{1}(X)_{\tor}^{\geom} \twoheadrightarrow \Coker(\rho_{X,0})_{\tor}$. We are now done because $\pi^{\ab}_{1}(X)_{\tor}^{\geom} \hookrightarrow \pi^{\ab}_{1}(X)_{\tor} $ and the later group is finite by Theorem~\ref{thm:A}. 
\end{enumerate}
This completes the proof of the theorem. \qed

%As a corollary of Theorem~\ref{thm:CF-Reg}, we prove a refined version of Theorem~\ref{thm:B'} in the case of regular curves (see assumptions below). 

\begin{cor}\label{cor:CF-Reg}
    Let $X$ be a geometrically integral, regular, projective curve over local field $k$ of positive characteristics $p > 0$. Assume that $X$ is generically smooth. Then 
\begin{center}
    $V(X/k) \cong D \oplus F,$
\end{center}
    where $D$ is a divisible group and $F$ is a finite group.
\end{cor}
 \begin{proof}
 We consider the commutative diagram \eqref{eqn:Struc-geom-*1.8} and note that its right vertical arrow is injective. 
It then follows that  $\Ker(\rho_{X,0}) = \Ker(\rho_X)$, which is divisible by Theorem~\ref{thm:CF-Reg}(1). By Theorem~\ref{thm:CF-Reg}$(2)$, \text{Image}$(\rho_{X,0})$ is finite and hence by letting $D=  \Ker(\rho_{X,0})$ and $F= \text{Image}(\rho_{X,0})$, we get the desired result. 
 \end{proof}

  \begin{comment}
     
 \begin{proof}
  
     This easily follows from Theorem~\ref{thm:CF-Reg}. Indeed, 
     considering the top
diagram on this page, we have $\Ker(\rho_{X,0}) = \Ker(\rho_X)$, since the map
$\rho_k$ is injective. Hence, by Theorem~\ref{thm:CF-Reg}$(1)$, the group
$\Ker(\rho_{X,0})$ is divisible. Moreover, \text{Image}$(\rho_{X,0})$ is finite by
Theorem~\ref{thm:CF-Reg}$(2)$. This implies that
\[
    V(X/k) \cong \Ker(\rho_{X,0}) \oplus \text{Image}(\rho_{X,0}),
\]
where the first summand is divisible and the second is a finite group. This completes the proof. 

  \end{proof}

\end{comment}

\section*{Acknowledgements}

We thank Professor Toshiro Hiranouchi for reading our manuscript and for his encouragement. 
The second author thanks IMSc, Chennai, for his visit in February 2024. The first author gratefully acknowledges HRI, Prayagraj, for their hospitality during his visit in June 2024, when much of this work was completed. The authors thank the
referees for reading the manuscript thoroughly and providing valuable suggestions
to improve its presentation.

\end{document}